\documentclass[11pt,a4paper]{article}
\usepackage[top=1in, bottom=1.25in, left=1in, right=1in]{geometry}
\usepackage[latin1]{inputenc}
\usepackage[english]{babel}
\usepackage{mathrsfs}
\usepackage{graphicx,epstopdf}
\usepackage{caption,subcaption,float}
\usepackage[dvipsnames]{xcolor}
\usepackage{multirow}
\usepackage{tikz,tikz-cd} 
\usepackage{amsthm}
\usepackage{tabls}
\usepackage{cite}
\usepackage[round,sort,comma]{natbib}
\usepackage{setspace}
\usepackage[ruled, linesnumberedhidden,vlined]{algorithm2e} 
\usetikzlibrary{matrix}
\usepackage{amsmath,amssymb,dsfont,mathtools}
\usepackage{pinlabel}
\usepackage[shortlabels]{enumitem}

\usepackage{hyperref}
\hypersetup{
    colorlinks=true,       
    linkcolor=blue,          
    citecolor=olive,       
    filecolor=magenta,      
    urlcolor=blue          
}
\usepackage{aliascnt}

\begin{document}

\title{A new family of infinitely braided Thompson's groups}
\date{\today }
\author{Julio Aroca and Mar\'{i}a Cumplido}


\maketitle
\theoremstyle{plain}
\newtheorem{theorem}{Theorem}

\newaliascnt{lemma}{theorem}
\newtheorem{lemma}[lemma]{Lemma}
\aliascntresetthe{lemma}
\providecommand*{\lemmaautorefname}{Lemma}

\newaliascnt{proposition}{theorem}
\newtheorem{proposition}[proposition]{Proposition}
\aliascntresetthe{proposition}
\providecommand*{\propositionautorefname}{Proposition}

\newaliascnt{corollary}{theorem}
\newtheorem{corollary}[corollary]{Corollary}
\aliascntresetthe{corollary}
\providecommand*{\corollaryautorefname}{Corollary}

\newaliascnt{conjecture}{theorem}
\newtheorem{conjecture}[conjecture]{Conjecture}
\aliascntresetthe{conjecture}
\providecommand*{\conjectureautorefname}{Conjecture}

\newtheorem*{question*}{Question}

\theoremstyle{remark}

\newaliascnt{claim}{theorem}
\newaliascnt{remark}{theorem}
\newtheorem{claim}[claim]{Claim}
\newtheorem{remark}[remark]{Remark}
\newaliascnt{notation}{theorem}
\newtheorem{notation}[notation]{Notation}
\aliascntresetthe{notation}
\providecommand*{\notationautorefname}{Notation}

\aliascntresetthe{claim}
\providecommand*{\claimautorefname}{Claim}

\aliascntresetthe{remark}
\providecommand*{\remarkautorefname}{Remark}

\newtheorem*{claim*}{Claim}
\theoremstyle{definition}

\newaliascnt{definition}{theorem}
\newtheorem{definition}[definition]{Definition}
\aliascntresetthe{definition}
\providecommand*{\definitionautorefname}{Definition}

\newaliascnt{example}{theorem}
\newtheorem{example}[example]{Example}
\aliascntresetthe{example}
\providecommand*{\exampleautorefname}{Example}


\def\autorefspace{\hspace*{-0.5pt}}
\def\sectionautorefname{Section\autorefspace}
\def\subsectionautorefname{Section\autorefspace}
\def\subsubsectionautorefname{Section\autorefspace}
\def\figureautorefname{Figure\autorefspace}
\def\subfigureautorefname{Figure\autorefspace}
\def\tableautorefname{Table\autorefspace}
\def\equationautorefname{Equation\autorefspace}
\def\Itemautorefname{item\autorefspace}
\def\Hfootnoteautorefname{footnote\autorefspace}
\def\AMSautorefname{Equation\autorefspace}

\newcommand{\quotient}[2]{{\raisebox{.2em}{$#1$}\left/\raisebox{-.2em}{$#2$}\right.}}
\def\Stab{{\rm Stab}}

\newcommand{\myref}[2]{\hyperref[#1]{#2~\ref*{#1}}}

\begin{abstract}

We present a generalization of the Dehornoy-Brin braided Thompson group~$BV_2$ that uses recursive braids. Our new groups are denoted by $BV_{n,r}(H)$, for all $n\geq 2,r\geq 1$ and $H \leq \mathcal{B}_n$, where $\mathcal{B}_n$ is the braid group on $n$ strands. We give a new approach to deal with braided Thompson groups by using strand diagrams. We show that $BV_{n,r}(H)$ is finitely generated if $H$ is finitely generated.
\medskip

{\footnotesize
\noindent \emph{2000 Mathematics Subject Classification. 20F65, 20F05, 20F36}.

\noindent \emph{Key words. Thompson groups, braid groups, strand diagrams, rewriting systems.}}

\end{abstract}
\section{Introduction}
\label{ref:intro}

The aim of this article is to define new families of Thompson-like groups that generalize the ones defined independently by \citep{Dehornoy} and \citep{Brin1,Brin2}. Thompson-like groups are based on the notion of cloning system, firstly defined by  \citep{WitzelZaremsky}. A cloning system on a family of groups $(G_n)_{n \in \mathbb{N}}$ is a set of axioms and maps acting on $G_n$ which allows to get a group $\mathscr{T}(G_*)$, called the generalized Thompson group for the cloning system or Thompson-like group (see \citealp{Zaremsky} and \citealp{WitzelZaremsky} for more details). The motivation for defining these new families is that Thompson-like groups have proven to be new examples of interesting groups. Richard Thompson used the celebrated groups $F, T$ and $V$ in 1965 to construct finitely presented groups with unsolvable word problems (see \citealp{CFP,Higman} for the definitions of $F,T$ and $V$). Other interesting families of Thompson-like groups (constructed in a different way) are branch groups \citep{BGS} and self-similar groups \citep{BGN}, which contain the first examples of groups of intermediate growth, like Grigorchuk's group \citep{Grigorchuk}.

\medskip

We will focus our attention on the Thompson-like groups coming from a cloning system on the family of groups $(\mathcal{B}_n)_{n \in \mathbb{N}}$, the Artin braid groups \citep{Artin1}. In \citep{Brin1,Brin2}, Brin uses this family  to define a braided Thompson group $V_{br}$, that we denote in our article as $BV_2$; one of the first examples of Thompson-like groups with no torsion, apart from the family of Higman-Thompson groups $F_n$, where $n \geq 2$ \citep{Higman}. The usual way to understand $BV_2$ is as the group of affine and orientation-preserving isotopies from the binary Cantor set $\mathfrak{C}_2$ to itself, using finite covers consisting of dyadic intervals. Brin proves that $BV_2$ is finitely presented by using a purely algebraic argument based on Zappa-Sz\'ep products.

\medskip

Using a similar approach, we present in this article new families of Thompson-like groups based on the previous ones. In the first place, we give the natural generalization of $BV_2$ by using $r$ copies of $n$-ary Cantor sets $\mathfrak{C}_n$ for all $r \geq 1$ and $n \geq 2$, obtaining all braided Thompson-like versions $BV_{n,r}$ of the well-known Higman-Thompson groups $V_{n,r}$, defined on \citep{Higman}. Next, we apply a recursive definition of braid (inspired by the definition of self-similar group) to construct a family of infinitely braided Thompson-like groups $BV_{n,r}(H)$, where $H \leq \mathcal{B}_n$. These families arise also from cloning systems, but we prefer to use a more `geometric' definition for the groups, as this allow us to prove also finite generation. Therefore, we leave the formal definition of cloning system for the interested reader, see \citep{SZ}, \citep{WitzelZaremsky} and \citep{Zaremsky}.

Moreover, it is also possible to define the infinitely braided versions of $BV_{n,r}$ and~$BV_{n,r}(H)$: $\widehat{BV_n}$ and~$\widehat{BV_n}(H)$. Note that all these families can be compiled on only two: $BV_{n,r}(H)$ and~$\widehat{BV_n}(H)$, as we can consider $H = Id$. 

\medskip

The first part of the article is devoted to prove that all these families are groups. For this aim we use a generalization of the Belk-Matucci theory of strand diagrams \citep{BelkMatucci,Aroca} and rewriting systems \citep{Newman} to give a bijection between elements of~$BV_{n,r}(H)$ and braided strand diagrams, a family of oriented graphs with labelled vertices (see Section~3.2 for all the details). As it is easier to prove that braided strand diagrams form a group, we obtain the desired result. Finally, we prove the main theorem of this article: 

\begin{theorem}
If $H \leq \mathcal{B}_n$ is finitely generated, the groups $BV_{n,r}(H)$ are finitely generated for every $r \geq 1$, $n\geq 2$.
\end{theorem}

For this purpose, we use a different and simpler approach than the one used by Brin. With the help of Higman-Thompson's groups $F_{n,r}$ and the properties of braids and diagrams, we generalize the idea of \citep{Higman} to give explicit generators for all these groups (see \autoref{theorem:fg}, Section~4.1.1 and Section~4.3).  

\medskip

To summarize, in \autoref{sec:2} we define $BV_{n,r}(H)$ and $\widehat{BV_n}(H)$ for every $r \geq 1$, $n \geq 2$ and $H \leq \mathcal{B}_n$. In \autoref{sec:3} we prove they are in fact, groups, by defining braided strand diagrams and applying the theory of rewriting systems to them. Finally, the aim of \autoref{sec:fg} is to prove that $BV_{n,r}(H)$ is finitely generated (if $H$ is finitely generated). We give explicit set of generators for $r = 1$ and $H=Id,\, \mathcal{B}_n$ and when $H$ is a standard parabolic subgroup of $\mathcal{B}_n$.


\section{The infinitely braided groups \texorpdfstring{$\boldsymbol{BV_n(H)}$}{BVnH}}
\label{sec:2}
In this section we define the main objects of this article: the braided versions of Higman-Thompson's groups $V_{n,r}$, $BV_{n,r}$; and its generalization, the family of infinitely braided groups  $BV_{n,r}(H)$. The fact that these are indeed groups, is proved in  \autoref{sec:3}.

\subsection{Descriptions of \texorpdfstring{$\boldsymbol{V_{n,r}}$}{Vnr}, \texorpdfstring{$\boldsymbol{BV_{n,r}}$}{BVnr} and \texorpdfstring{$\boldsymbol{\widehat{BV_n}}$}{hatBVn}}

Let $\mathfrak{C}_n$ be the $n$-adic Cantor set, which is constructed inductively as follows: $\mathfrak{C}_n^1$ corresponds to first subdividing $\mathfrak{C}_n^0= [0,1]$ into $2n-1$ intervals of equal length, numbered $1, \ldots, 2n-1$ from left to right, and then taking the collection of odd-numbered subintervals. We will renumber these intervals from left to right and denote them $C^1_1,\dots, C^1_n$. Next, $\mathfrak{C}_n^2$ is obtained from $\mathfrak{C}_n^1$ by applying the same procedure to each interval $C^1_i$ to obtain $C^2_{(i-1)n +1}, C^2_{(i-1)n +2},\dots, C^2_{(i-1)n +n}$. 

\smallskip

We recursively define every $\mathfrak{C}_n^j$ with $j>1$ and its intervals. Let $\mathfrak{C}_n$ be the intersection of all~$\mathfrak{C}_n^i$. The elements of the Thompson's group $V_n$ are defined using covers of~$\mathfrak{C}_n$ by pairwise disjoint intervals of the form~$C_i^j$ chosen from any~$\mathfrak{C}^j_n$. For any pair of covers $C$ and $C'$ with the same number of intervals, we define an affine and orientation preserving map from the elements of $C$ to the elements of~$C'$. Then we restrict the map to $\mathfrak{C}_n$. This restriction is a homeomorphism of~$\mathfrak{C}_n$. Finally, we define~$V_n$ as the set of all maps of this kind, which turns to be a group under composition. 

\smallskip

The elements of $V_n$ are coded by pairs of finite full $n$-ary trees together with a bijection $\tau$ between their leaves. An example of such a pair is depicted on \autoref{fig:elementVn}, where the left (resp. the right) tree indicates how the first (resp. the second) cover is split. These trees are respectively called \textit{domain tree} $T$ and \textit{range tree}~$T'$. Therefore, any element $v \in V_n$ is represented as a triple $(T, \tau, T')$. Notice that this representative is not unique. A well known subgroup of $V_n$, that will be used in the last section, is $F_n$, consisting of all elements represented by triples $(T, \tau, T')$ where $\tau$ is trivial.

\medskip

The set of leaves of all possible $n$-ary trees is in bijection with the set of finite words on the alphabet $\mathcal{A}_n = \{0, \dots, n -1\}$, denoted by $\mathcal{A}^*_n$. The word assigned to each leaf depends on the path taken from the root to reach the leaf. For example, observe that in \autoref{fig:elementVn} the set of leaves of the first tree is $\{00, 01,020,021,022,1,2\}$. This labelling induces a natural order on the set of leaves of a tree. By abuse of notation, we will say that a word in $\mathcal{A}^*_n$  is its represented leaf.

Let $v \in V_n$ be represented by a tree pair $(T, \tau, T')$ such that $T$ and $T'$ have $l$ leaves. Note that there are two numbers assigned on each leaf: one of them is the coordinate of the leaf, that is, a finite word on the alphabet $\{0, \dots, n -1\}$. The other one is a number in $\{1, \dots, l\}$ depending on the bijection $\tau$.

\begin{definition}
Let $T$ be finite full $n$-ary tree. We define a \emph{caret} as a subtree of $T$ consisting of a set of leaves of the form $\{w0, w1, \dots, w(n-1)\}$, the vertex $w$ and the set of edges linking $w$ with $wi$ for all $i \in \{0,\dots, n-1\}$, for any finite word $w \in \mathcal{A}^*_n$. We say that a caret is $\emph{final}$ if its set of leaves is also a set of leaves of $T$. We represent a caret by using its set of leaves, so we may omit $\{w\}$. 
\end{definition}

\noindent
As example, in \autoref{fig:elementVn} the set of leaves $\{00,01,02\}$ of $T$ is a caret, and $\{020,021,022\}$ of $T$ is a final caret. 

\begin{definition}
Let $T$ be a full finite $n$-ary tree. Let $w \in \mathcal{A}^*_n$ be a leaf of $T$. We denote by $T[w]$ the tree obtained from $T$ by appending a final caret to $w$. Similarly, we define $T[c]^{-1}$ as the tree obtained from $T$ by removing a specific final caret $c = \{w0, \dots, w(n-1)\}$ from it. We say that $T[w]$ (resp. $T[c]^{-1}$) is an \emph{expansion} (resp. \emph{reduction}) of $T$.
\end{definition}

\noindent
Keep in mind that, in a composition of expansions $T[w][w']$, $w'$ must always be a leaf of~$T[w]$, although it does not need to be a leaf of~$T$.
 
\medskip
 
There exist infinitely many triples  $(T, \tau, T')$ which define the same element of $V_n$. Let $v = (T,\tau,T')$ be a homeomorphism of $\mathfrak{C}_n$ where a cover $c \in C$ is mapped to a cover $c' \in C'$. We can consider the subdivision of both $c$ and $c'$ into $n$ pieces $c_1, \dots, c_n$ and $c'_1, \dots, c'_n$ such that the affine map takes $c_i$ to $c'_i$ for all $i\in \{1, \dots, n\}$. The corresponding tree-pair representative of the subdivided coverings leads to the same element $v$, previously defined in terms of~$C$ and~$C'$. In terms of trees, we add a final caret on the leaves $w$ and $w'$ representing the intervals $c$ and $c'$ respectively. Thus we obtain a new triple $v = (T[w], \tau',T'[w'])$ where $\tau'$ is the corresponding bijection including the new set of leaves. 

\medskip
\noindent
We say that a tree-pair representative of an element is \emph{reduced} if the number of covers of both~$C$ and~$C'$ (that is, the number of leaves of~$T$ and~$T'$) is minimal. A simple way to distinguish non-reduced elements is by checking if there exists a final caret  $\{w0, w1, \dots, w(n-1)\}$ which is mapped to another final caret $\{w'0, w'1, \dots, w'(n-1)\}$ such that $wi$ is mapped to $w'i$ for every $i \in \{0,\dots,n-1\}$.

\begin{figure}[h]
\centering
\labellist
\pinlabel $1$ at 5 23
\pinlabel $2$ at 12.5 23
\pinlabel $3$ at 12.5 -3
\pinlabel $4$ at 20 -3
\pinlabel $5$ at 28 -3
\pinlabel $6$ at 28.5 50
\pinlabel $7$ at 45 50

\pinlabel $3$ at 63 23
\pinlabel $2$ at 70.5 23
\pinlabel $7$ at 78.5 23
\pinlabel $4$ at 96 23
\pinlabel $1$ at 103.5 23
\pinlabel $6$ at 111.5 23
\pinlabel $5$ at 87 50

\endlabellist
\centering
\includegraphics[scale=1.3]{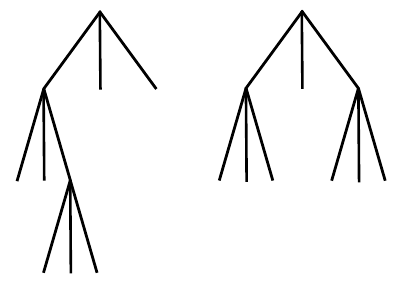}
\caption{An element of $V_3$.}
\label{fig:elementVn}
\end{figure}
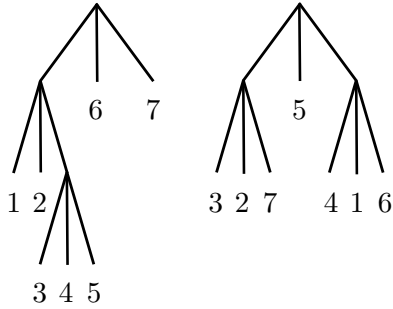

Finally, the way to compose elements of $V_n$ is as follows: given two elements $v = (T, \tau, T')$ and $w = (T', \tau',T'')$ the composition is the element $vw = (T, \tau\tau', T'')$. Note that the range tree of~$v$ and the domain tree of~$w$ must be equal. If not, it is possible to expand the trees of both~$v$ and~$w$ until we get tree-pair representatives which can be composed.

\medskip

In a similar way, we can consider $r$ copies of $\mathfrak{C}_n$ instead of only one, obtaining the group~$V_{n,r}$. In that case, elements are represented as a pair of forests of $r$ finite full $n$-ary trees together with a bijection $\tau$ between their leaves. As before, if $\tau = Id$, we obtain groups $F_{n,r}$. See \citep{CFP,Higman} for an introduction on these groups. 

\medskip

The previous description of $V_n$ can be easily extended to obtain the braided Thompson's groups $BV_n$. Braided Thompson's group $BV_2$ was introduced in \citep{Brin1,Brin2} and \citep{Dehornoy}. As before, given any pair of covers $C$ and $C'$ of $\mathfrak{C}_n$ with the same number $m$ of elements, we embed them in $\mathbb{R}\times \{1\}$ and $\mathbb{R} \times \{0\}$ respectively. Then, we define an orientation-preserving isotopy with compact support from the elements of $C$ to the elements of~$C'$. In this case, the isotopy is represented by a braid~$\beta$ with $m$~strands.

\begin{definition}

A \emph{braid} is a collection of $m$~disjoint paths in a cylinder connecting $m$ points of its upper disk to $m$~points of its lower disk and running monotonically in the vertical direction (see \autoref{fig:elementBVn}). Two braids $\beta$, $\beta'$ are equivalent if we can continuously deform $\beta$ into $\beta'$ without intersecting the paths. The equivalence classes of these objects (that will be called braids by an abuse of notation) are the elements of the \emph{braid group} with $m$ strands, $\mathcal{B}_m$, introduced in \citep{Artin1}, which is presented as follows:
\[\mathcal{B}_m=\{\sigma_1,\dots,\sigma_{m-1} \, | \, \sigma_i \sigma_j =\sigma_j \sigma_i \text{ if } |i-j|>1,\, \sigma_i\sigma_j\sigma_i =\sigma_j\sigma_i\sigma_j \text{ if } |i-j|=1.\}\]
Here $\sigma_i$ (resp. $\sigma_i^{-1}$) is the braid in which the strand in the $i$-th position passes over (resp. under) the strand in the $(i+1)$-th position. The set of $\sigma_i$'s is called the set of \emph{Artin generators.} 
\end{definition}

A representative of an element $v \in BV_n$ is a pair of finite full $n$-ary trees together with a braid $\beta$ between their leaves, that is, $v = (T,\beta,T')$. This triple is called \emph{braided tree-pair}. The composition of these elements works as before. In order to better understand braided tree-pairs, we will use \emph{braided diagrams} in which the range tree is pictured upside down below the domain tree, as pictured in  \autoref{fig:elementBVn}. These diagrams will be thoroughly used in this paper. 

\medskip
\noindent
In a similar manner, by considering a finite number $r$ of copies of $\mathfrak{C}_n$ and using the same definition as above, we get the group~$BV_{n,r}$. 

\begin{figure}[h]
\centering
\labellist
\pinlabel $1$ at 202 160
\pinlabel $2$ at 218 160
\pinlabel $3$ at 234 160
\pinlabel $4$ at 251 160
\pinlabel $5$ at 268 160

\pinlabel $\sigma_3$ at 280 134
\pinlabel $\sigma_2$ at 280 115
\pinlabel $\sigma^{-1}_1$ at 283 96
\pinlabel $\sigma^{-1}_4$ at 283 70
\pinlabel $\sigma^{-1}_3$ at 283 50

\endlabellist
\includegraphics[scale=1]{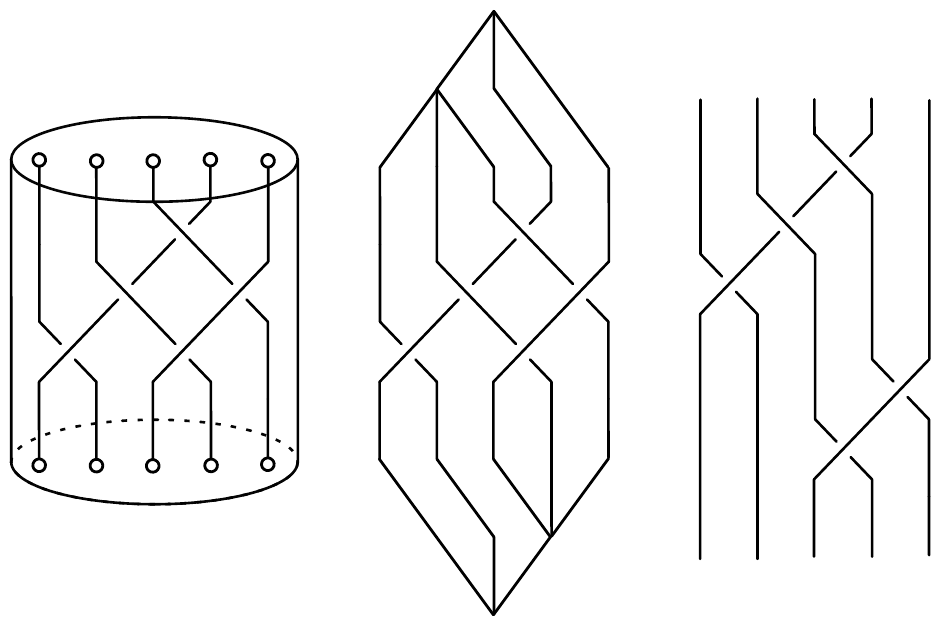}
\caption{An element of $BV_3$ with its braid $\beta = \sigma_3\sigma_2\sigma^{-1}_1\sigma^{-1}_4\sigma^{-1}_3$.}
\label{fig:elementBVn}
\end{figure}
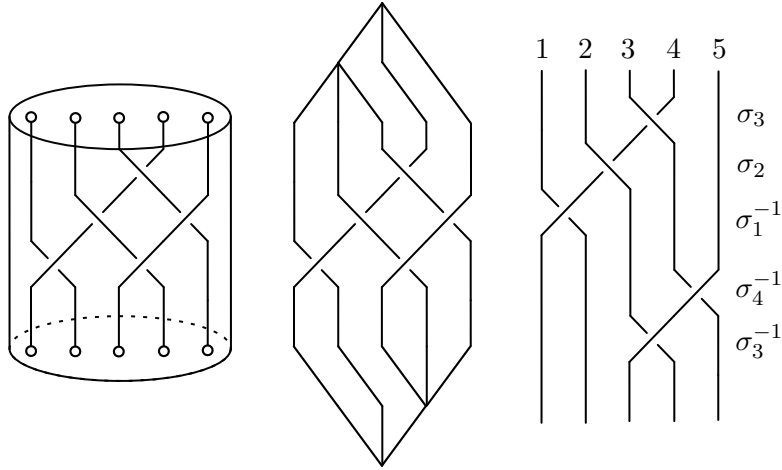

\bigskip

The group  $\widehat{BV_n}$ is built by embedding a countable number of copies of $\mathfrak{C}_n$, one on each interval $[2i,2i+1]\times\{1\}$ and $[2i,2i+1]\times\{0\}$ of $\mathbb{R}\times \{1\}$ and $\mathbb{R}\times \{0\}$ respectively. Consider two infinite covers $C$ and $C'$ such that the intervals are pairwise disjoint, as before, and all but finitely many of them are of the form $[2i,2i+1]\times\{1\}$ for $C$ (resp. of the form $[2i,2i+1]\times\{0\}$ for~$C'$). Finally, we map $C$ to $C'$ by an isotopy of $\mathbb{R}^2$, such that the images of the chosen intervals are parallel to the $x$-axis. Since the number of intervals is infinite, the isotopy could be a ``shift'' taking place for large values of $x$. We impose that a shift must be done by an isotopy of the form $$(x,y) \rightarrow (x+td(1- \vert y  \vert ),y) $$
outside a compact, for $\vert y \vert < 1$, $x > K$ for some positive constant $K$; and by the identity otherwise. The integer $d$ is the total amount of
shift and $t$ is the parameter of the isotopy, see \citep{Brin1}. An example of such an element is depicted on \autoref{fig:elementwBVn}.

\begin{figure}[h]
\centering
\includegraphics[scale=1]{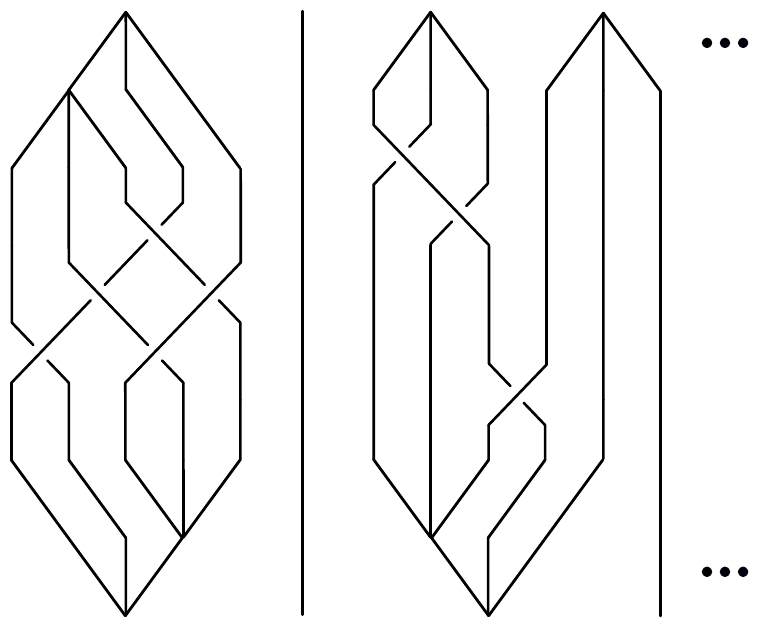}
\caption{An element of $\widehat{BV_3}$.}
\label{fig:elementwBVn}
\end{figure}

\bigskip
\noindent
The proof of the fact that $BV_2$ and $\widehat{BV_2}$ are groups can be found in \citep{Brin1}. In  \autoref{sec:3}, we will prove the same for our groups, as well as for the ones defined hereunder. Finally, the infinite versions $\widehat{V_n}$ and $\widehat{F_n}$ are defined in the same way as before, see \citep{Brown}.

\subsection{Descriptions of \texorpdfstring{$\boldsymbol{BV_n(H)}$}{BVnH} and \texorpdfstring{$\boldsymbol{\widehat{BV_n}(H)}$}{hatBVnH}}

Let $H$ be a subgroup of the group of braids on $n$ strands $\mathcal{B}_n$. The aim of this paper is to define new groups $BV_n(H)$ as a generalization of $BV_n$.

\medskip

Let $C$ and $C'$ be two covers of $\mathfrak{C}_n$ with the same number of elements. Let $c \in C$, $c' \in C'$, and $h \in H \leq \mathcal{B}_n$. A \textit{recursive braid} of type $h$ between $c$ and $c'$ is a braid with infinitely many strands obtained by the following process: Replace $c$ and $c'$ by two subcovers $c_1, \dots, c_n$ and $c'_1,\dots,c'_n$ respectively, such that~$h$ is an isotopy from $c_i$ to $c'_j$, (that is, $c_1, \dots, c_n$ and $c'_1,\dots,c'_n$ are braided by~$h$). We repeat this process on every $c_i$ and $c'_j$ by subdividing them and applying~$h$ again, and so on. 
In \autoref{fig:rcilindro}, we can see graphically how to construct a recursive braid. We define the composition of a recursive braid of type $h_1$ between $c$ and $c'$ with a recursive braid of type $h_2$ between $c'$ and $c''$ as the recursive braid of type $h_1h_2$ between $c$ and $c''$.
 
 \begin{figure}[h]
\centering
\labellist
\pinlabel $\sigma_1$ at 60 83
\pinlabel $\sigma_1$ at 32 48
\pinlabel $\sigma_1$ at 68 48

\endlabellist
\centering
\includegraphics[scale=1.3]{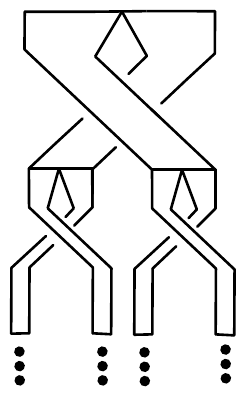}
\caption{An example of recursive braid where $h = \sigma_1$.}
\label{fig:rcilindro}
\end{figure}

\medskip 
 
The group $BV_n(H)$ is the group of elements in $BV_n$ together with (possibly) recursive braids between covers of $\mathfrak{C}_n$. The way to represent an element $v \in BV_n(H)$ is as a triple $(T,\beta,T') \in BV_n$ together with a set $\lambda=\{h_i\}_{i=1}^m, h_i \in H \leq \mathcal{B}_{n}$, where $m$ is the number of leaves of both~$T$ and~$T'$. This set corresponds to the set of recursive braids on the $m$~strands of~$\beta$. Each label indicates that there is a recursive braid of type $h_i$ between the $\beta(i)^{-1}$-th and $i$-th intervals linked by the corresponding strand of~$\beta$. Therefore, we write $v = (T,\beta,\lambda,T')$. As before, there are infinitely many tree pair representations of the same element $v \in BV_n(H)$. In this case, final carets are mapped by elements of $H$: if some $c\in C$ is mapped to $c' \in C'$ with a recursive braid $h$, then the same isotopy maps the subcovering $c_i$ to $h(c_i) = c'_{j(h)}$ with a recursive braid $h$ for all $i \in \{1,\dots, n\}$, where~$j$ is the permutation on $n$ elements induced by~$h$. In terms of tree-pair representations, $v$ maps a final caret to another one following the rules of $h$. Finally, note that elements of~$BV_n$ can be expressed in terms of recursive braids, where $\lambda = \{Id, \dots, Id\}$. Therefore, $BV_n= BV_n(Id)$.

\medskip
\noindent
The group $\widehat{BV_n}(H)$ is defined in the same way as $\widehat{BV_n}$, by adding recursive braids. Similarly, if we consider a finite number $r$ of copies of $\mathfrak{C}_n$, we obtain $BV_{n,r}(H)$. 


\section{Braided diagrams and rewriting systems}
\label{sec:3}

In this section we prove that all the previously defined groups are in fact groups. In order to do that, we use a generalization of the theory of strand diagrams \citep{BelkMatucci,Aroca} and rewriting systems \citep{Newman}. This section is heavily inspired on the aforementioned work.

\subsection{Basics about graphs} Consider a directed graph $\Gamma$. Let $V(\Gamma)$ be its set of (possibly labelled) vertices, and let $E(\Gamma) = \{e_1,e_2,\dots, e_s\} \subset V(\Gamma) \times V(\Gamma)$ be the set of oriented edges, which have the form $e = (v,v')$ or $(v',v)$ depending on the orientation. An \textit{oriented path} is a sequence of oriented edges $\{e_1,...,e_t\} = \{(v_{i(1)},v_{j(1)}),\dots, (v_{i(t)},v_{j(t)})\}$ such that $v_{j(k)} = v_{i(k+1)} \, \forall k \in \{1,\dots, t-1\}$. If $v_{j(t)} = v_{i(1)}$, we have an \textit{oriented loop}.

\begin{definition}
The \textit{degree} of a vertex $v \in V(\Gamma)$ is the number of edges which have $v$ as endpoint, that is, those which have the form $(v,v')$ or $(v',v)$ for some $v' \in V(\Gamma)$. If an edge has the form $(v,v)$, $v$ has degree $2$.
\end{definition}

\begin{definition}
A vertex $v$ is a \textit{source} (resp. a \textit{sink}) for a finite set of directed edges if they have~$v$ as starting point (resp. ending point). A vertex~$v$ is a \textit{main source} (resp. a \textit{main sink}) if it is a source (resp. a sink) of degree one. 

Let $n \geq 2$. A \textit{split} (resp. a \textit{merge}) is a vertex $v$ of degree $n+1$ which is a sink for one edge and a source for the others (resp. a source for one edge and a sink for the others). A \textit{white vertex} $v^h$ is a vertex of degree $2$ with a label $h \in H \leq \mathcal{B}_n$.
\end{definition}

\begin{definition}
We say that $\Gamma$ is \textit{acyclic} if the graph has no oriented loops. From now on, $\Gamma$ will be a directed acyclic graph.
\end{definition}

\begin{definition}
A \textit{pitchfork graph} is a graph whose vertices are only main sources, main sinks, white vertices, splits or merges for a fixed $n \geq 2$. 
\end{definition}

\subsection{Braided diagrams} 
Let $g \in BV_n(H)$ for some $n \geq 2$, $H \leq \mathcal{B}_n$ and $({T}, \beta, \lambda, {T}')$ a braided tree-pair representative of~$g$. We recall that we can construct a braided diagram representation $\Gamma'$ from an element $({T}, \beta, {T}')$ of $BV_n$ by picturing $T'$ upside down below $T$ and joining the leaves of both trees with the braid $\beta$ (see again \autoref{fig:elementBVn}). 
The aim if this subsection is to construct a well-defined braided diagram for $g$.

\smallskip
Notice that $\Gamma'$ may not be planar, but 3-dimensional. However, we can consider the obvious planar projection, depicted in \autoref{fig:elementBVn}, that will be called, by abuse of notation, braided diagram.

\begin{definition}
We say that an oriented path of an acyclic graph is a \emph{strand} if it only contains white vertices (excluding its endpoints).
\end{definition}

\begin{definition}
\label{def:braid}
A \textit{braided diagram} is a planar projection $\Gamma:=p(\Gamma')$ of a finite directed acyclic pitchfork graph $\Gamma'$ satisfying the following properties: 

\begin{itemize}

\item Two vertices of $\Gamma'$ never have the same image;

\item No vertex is mapped onto an edge that it is not an endpoint of;

\item The images of the edges of $\Gamma'$ intersect in a finite set of points;

\item The image of a strand cannot intersect itself.

\end{itemize}
The set of edge intersections that are not endpoints is the \emph{set of crossings} $C(\Gamma)$ of $\Gamma$.

\end{definition}

To distinguish isomorphic planar graphs with different crossings, we impose a rotation and a crossing system:

\begin{definition}
Let $\Gamma$ be a braided diagram. A \textit{rotation system} of $\Gamma$ is a map $\rho_{\Gamma}: E(\Gamma) \longrightarrow \{0, \dots, n\}^2$ which gives an order to every edge of $\Gamma$ around its endpoints as follows:
\begin{enumerate}
\item A counter-clockwise order to the directed edges of a split $v$, where the $0$-th edge is the edge which has $v$ as sink.
\item A clockwise order to the directed edges of a merge $v$, where the $0$-th edge is the edge which has $v$ as source.
\item A $0$ to the edge which has a white vertex $v^h$ as sink, and a $1$ to the one which has $v^h$ as source. 
\end{enumerate}
\end{definition}

\begin{definition}
Let $\Gamma$ be a braided diagram. A \textit{crossing system} of $\Gamma$ is a map $\kappa_{\Gamma}: C(\Gamma) \longrightarrow \{1, -1\}$ which gives a label to every crossing of $\Gamma$. Geometrically, the crossings are depicted as in  \autoref{fig:crossings}.
\end{definition}

\begin{figure}[h]
\centering
\includegraphics[scale=1.3]{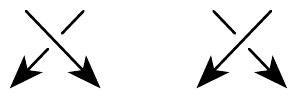}
\caption{The positive and negative crossings.}
\label{fig:crossings}
\end{figure}

\noindent
Observe that $\rho_{\Gamma}$ and $\kappa_{\Gamma}$ completely determine the crossings of $\Gamma$ and will represent the Artin generators composing a braid and their inverses. From now on, we will always consider that a braided diagram is endowed with a crossing and a rotation system. 

\begin{remark}
Note that these systems provide a natural injection $f:\Gamma \cup C(\Gamma)\rightarrow \mathbb{R}^3$. By abuse of notation, we will denote $f(\Gamma \cup C(\Gamma))$ by $f(\Gamma)$. $\Gamma$ will be called a \emph{natural projection} of $f(\Gamma)$.
\end{remark}

\begin{definition}
\label{def:equal}
 Two braided diagrams $ \Gamma_1$ and $\Gamma_2$ are equal if there exists an isomorphism $\phi:\Gamma_1 \rightarrow \Gamma_2$ such that:
\begin{enumerate}
\item $\phi(\rho_{\Gamma_1}) = \rho_{\Gamma_2}$, 
\item $\phi(\kappa_{\Gamma_1}) = \kappa_{\Gamma_2}$, and
\item $\phi(v^{h}) = \phi(v)^{h}$, $\forall v \in V(\Gamma_1)$. 
\end{enumerate}  
for every white vertex $v^h$. 
\end{definition} 

\noindent
It is possible to compose two braided diagrams $\Gamma'$ and $\Gamma''$ if the number of main sinks of $\Gamma'$ is equal to the number of main sources of $\Gamma''$. We then identify from left to right the main sinks of $\Gamma'$ with the main sources of $\Gamma''$ without creating any new crossing. The \textit{composition} is again a braided diagram.  

\medskip
We now explain how to construct the {braided diagram} for an element~$g \in BV_n(H)$, with tree-pair representative $({T}, \beta, \lambda, {T}')$. Take the braided diagram of $({T}, \beta, {T}')$  and then append to the $\beta(i)$-th leaf of~$T$ the corresponding white vertex~$v^{h_i}$. Next, append to the roots of both~${T}$ and~${T}'$ an edge and a vertex. Give the orientation from the vertex appended in ${T}$ to the leaves of~${T}$, from the leaves of ${T}$ to the leaves of ${T}'$ and finally from the leaves of ${T}'$ to the appended vertex of ${T}'$.  The resulting directed acyclic pitchfork graph is the \textit{braided diagram} of $g$.  The process is similar for elements of $BV_{n,r}(H)$ and $\widehat{BV_n}(H)$, but using forests instead of trees. Note that not all braided diagrams are obtained as a consequence of this process.

\subsection{Rewriting systems and confluence}

As said on the beginning of this section, we use the theory of rewriting systems \citep{Newman} in order to prove that there exists a bijection between equivalence classes of braided diagrams and elements of $BV_n(H)$.

\begin{definition} Let $\Gamma'$ be a finite directed acyclic pitchfork graph. Suppose that  $\Gamma'$ has a collection of strands $\mathcal{S}$ that has a neighbourhood isotopic to a cylinder which does not contain other edges. If, up to isotopy, this cylinder defines a braid $\alpha$, we say that~$\mathcal{S}$ is a \emph{sub-braid}~$\alpha$ of~$\Gamma'$.
For example, in \autoref{fig:elementBVn}, the collection of the second and third strands of $\beta$ defines a trivial sub-braid. 
\end{definition}

\begin{definition}
\medskip
\noindent
Let $\Gamma'=f(\Gamma)$ be the natural injection of a braided diagram $\Gamma$ in $\mathbb{R}^3$. We define the following \emph{moves} on $\Gamma$:

\begin{itemize}

\item[(1)] Consider a split $v_1$ and a merge $v_2$ of $\Gamma$ (and $\Gamma'$).  Suppose that $\Gamma'$ has a sub-braid $h$ on $n$~strands connecting every $i$-th edge of $v_1$ with the $h(i)$-th edge of $v_2$. Also suppose that each strand of $h$ contains a (possibly empty) sequence of white vertices $\{v^{h_1}, \dots, v^{h_{t(i)}}\}$ with $ h_1 \circ \dots \circ h_{t(i)} = h$, where each $t(i)$ depends on the strand.

Take the connected subgraphs $\Gamma_0$ of $\Gamma$ and $\Gamma'_0$ of $\Gamma'$ formed by $v_1$, $v_2$, $h$, the edges starting at $v_1$ and the edges arriving at $v_2$. Replace a neighbourhood of $\Gamma'_0$ by the neighbourhood  of a single strand $\Gamma''_0$ containing a white vertex $v^h$ if $h \neq Id$; or a single strand with no white vertices, if $h = Id$. Let $ \Gamma''$ be the new pitchfork graph. A \emph{$1$-move} replaces $\Gamma$ by a natural projection $p(\Gamma'')$ such that $p(\Gamma''\setminus \Gamma_0'')= \Gamma \setminus \Gamma_0 $. See \autoref{fig:wtypeI}.   



\item[(2)] 

Consider a merge~$v_1$ and a split~$v_2$ of  $\Gamma$ (and $\Gamma'$). Suppose that $v_1$ and $v_2$ are connected by a strand $d$. Let $\{v^{h_1}, \dots, v^{h_{t}}\}$ be a (possibly empty) sequence of white vertices on this strand with $h_1 \circ \dots \circ h_{t} = h$. 

Take the connected subgraphs $\Gamma_0$ of $\Gamma$ and $\Gamma'_0$ of $\Gamma'$ formed by $v_1$, $v_2$ and $d$. Replace~$\Gamma_0'$ with the neighbourhood of the subgraph $\Gamma_0''$ obtained by adding a white vertex~$v^{h}$ to every strand of~$h$. This braid connects the $i$-th edge of~$v_1$ with the $h(i)$-th edge of~$v_2$, for all $i \in \{1,\dots,n\}$. If $h = Id$, then we replace the neighbourhood of $\Gamma_0$ with the neighbourhood of $n$~edges connecting the $i$-th edge of~$v_1$ and the $i$-th edge of~$v_2$, with no crossings between them. Let $ \Gamma''$ be the new pitchfork graph. A \emph{$2$-move} replaces $\Gamma$ by a natural projection $p(\Gamma'')$ such that $p(\Gamma''\setminus \Gamma_0'')= \Gamma \setminus \Gamma_0 $. See \autoref{fig:wtypeI}.



\begin{figure}[h]
\centering
\labellist
\pinlabel $h$ at 0 143
\pinlabel $h$ at 27 143
\pinlabel $h$ at 75 143
\pinlabel $h$ at 105 143
\pinlabel $h$ at 152 143
\pinlabel $h$ at 179 143

\pinlabel $h_1$ at -3 75
\pinlabel $h_1$ at 29 75
\pinlabel $h_2$ at -3 63
\pinlabel $h_2$ at 29 63
\pinlabel $h_n$ at -3 26
\pinlabel $h_n$ at 29 26

\pinlabel $h_1$ at 122 68
\pinlabel $h_2$ at 122 58
\pinlabel $h_n$ at 122 34

\pinlabel $h$ at 75 52
\pinlabel $h$ at 150 52
\pinlabel $h$ at 179 52
\endlabellist
\centering
\includegraphics[scale=1.2]{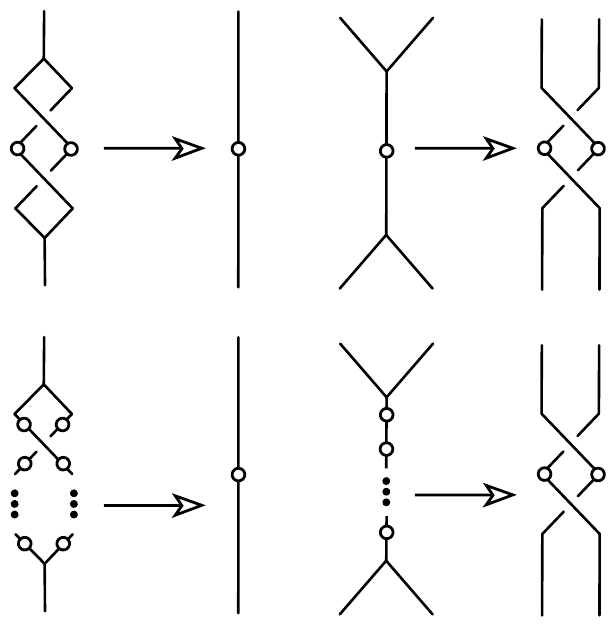}
\caption{Examples of $1$-moves and $2$-moves, where $h_1 \circ h_2 \circ \dots \circ h_n =h$.}
\label{fig:wtypeI}
\end{figure}
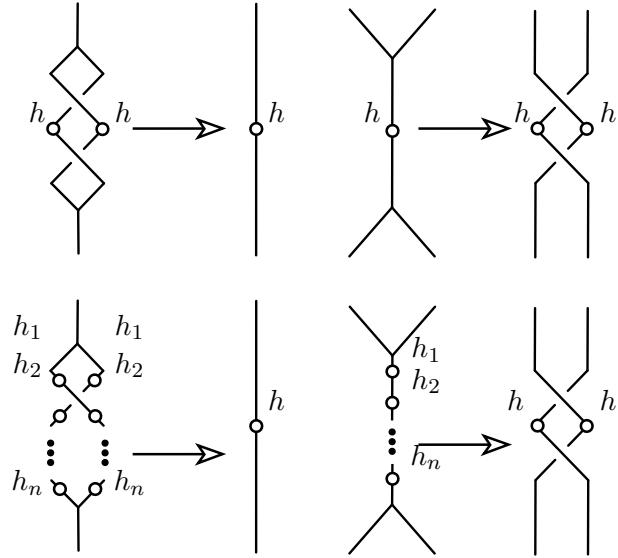

\item[(3)] Let $\Gamma_0$ be the subgraph of $\Gamma$ consisting of a crossing between two edges, such that one of them contains a white vertex. Then a \emph{$3$-move} moves the white vertex along the crossing, as \autoref{fig:wtypeII} shows.

\item[(4)]  Let $\Gamma_0$ be the subgraph of $\Gamma$ consisting of one edge whose endpoints are two white vertices~$v^{h_1}$ and~$v^{h_2}$. A \emph{$4$-move} replaces the neighbourhood of~$\Gamma_0$  with the neighbourhood of a strand containing a single white vertex $v^{h_2 \circ h_1}$. If $h_2 \circ h_1 = Id$, it replaces the neighbourhood of~$\Gamma_0$ with a strand with no white vertices. See \autoref{fig:wtypeII}.

\item[(5)] Suppose that in $\Gamma$ we have a crossing between two strands such that one of them ends on a split. Then a \emph{$5$-move} pushes the crossing along the split, such that the strand without the split crosses the $n$ strands of the split as in \autoref{fig_typeII2}. The same occurs for the symmetric case involving two strands and one merge.

\item[(6)] Let $\Gamma_0$ be the neighbourhood of the subgraph of $\Gamma$ consisting of an edge $e$ whose endpoints are a merge and a white vertex $v^h$, where $e$ is the $0$-th edge of the merge. Consider $\Gamma_1$, obtained from $\Gamma_0$ by performing the inverse of a type $1$-move on $e$. Let~$\Gamma_2$ be obtained from~$\Gamma_1$ by performing a $2$-move on the merge and the new split created by the previous reduction. A \emph{$6$-move} replaces $\Gamma_0$ with $\Gamma_2$ (see \autoref{fig_typeII2}). The symmetric case (a split whose $0$-th edge has a white vertex as endpoint) is analogously treated. There are several cases contained here, as the inverse of a $1$-move is not unique.

\end{itemize}

\end{definition}

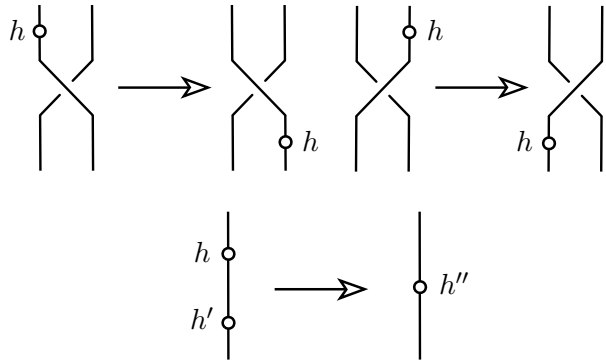
\begin{figure}[h]
\centering
\labellist

\pinlabel $h$ at -2 107
\pinlabel $h$ at 129 107
\pinlabel $h$ at 90 72
\pinlabel $h$ at 157 72
\pinlabel $h$ at 56 37
\pinlabel $h'$ at 56 16
\pinlabel $h''$ at 135 27
\endlabellist
\centering
\includegraphics[scale=1.2]{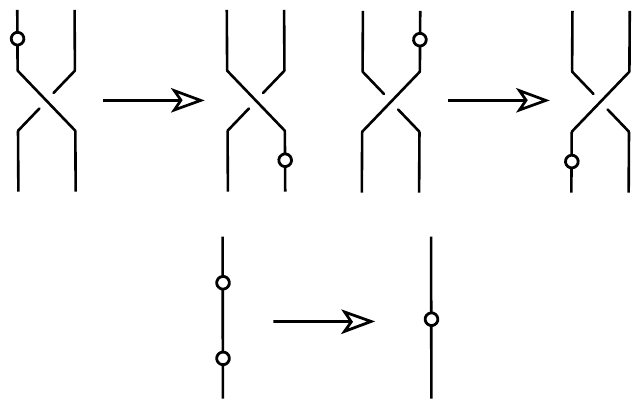}
\caption{Examples of $3$-moves and $4$-moves, where $h'\circ h = h''$.}
\label{fig:wtypeII}
\end{figure}

\begin{figure}[h]
\centering

\labellist
\pinlabel $h$ at 25 41
\pinlabel $h$ at 78 49
\pinlabel $h$ at 106 49

\pinlabel $h$ at 163 39
\pinlabel $h$ at 215 27
\pinlabel $h$ at 243 27

\endlabellist
\includegraphics[scale=1.3]{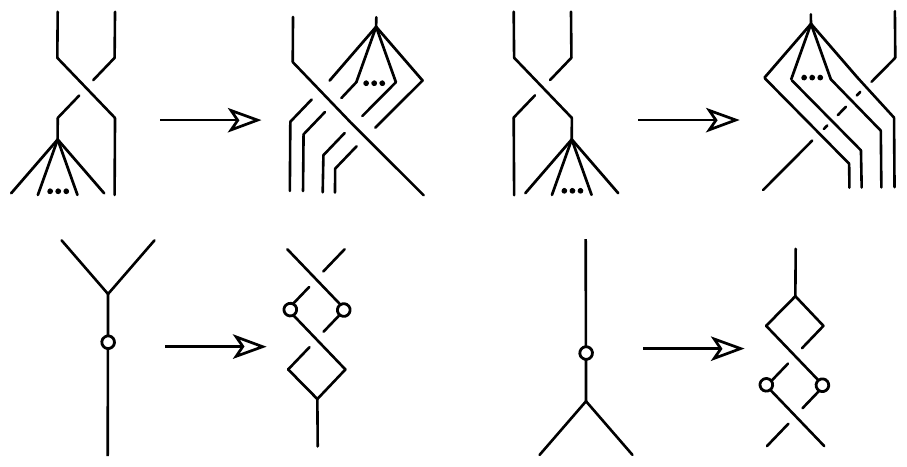}
\caption{Examples of $5$-moves and $6$-moves.}
\label{fig_typeII2}
\end{figure}

\begin{definition}
\label{def:equiv}
For two braided diagrams $\Gamma$, $\Gamma'$, we say that $\Gamma'$ is a \emph{reduction} of $\Gamma$ if there exists a sequence of moves which takes $\Gamma$ to $\Gamma'$. Two braided diagrams are \emph{equivalent} if one is a reduction of the other. 
\end{definition}

\begin{definition}
\label{def:red}
A braided diagram is \emph{reduced} if no moves can be performed on it. 
\end{definition}

We construct a directed graph $R$ from the set of all braided diagrams as follows: the vertex set of $R$ consists of the set of all braided diagrams. The vertices are called \emph{states}. We have an oriented edge from a state $s$ to a state $s'$ if we obtain $s'$ from $s$ by performing a move. The graph $R$ is called a \emph{rewriting system}.

\begin{definition}
A rewriting system $R$ is \textit{terminating} if every oriented path of $R$ has finite length. A state $s$ is \textit{reduced} if no oriented paths start from $s$.
\end{definition}

\noindent
Note that if a rewriting system is terminating, then every state has a reduced form. 

\begin{definition} \label{def:confluence}
We say that a rewriting system is \textit{locally confluent} if for all triples of states $s_0,s_1,s_2$ such that $s_1$ and $s_2$ are reductions of $s_0$, there exists a state $s_3$ which is a reduction of both $s_1$ and $s_2$.
\end{definition}

\begin{proposition}
\label{prop:equiv}
Every braided diagram is equivalent to a unique reduced braided diagram.
\end{proposition}

\begin{proof}
As it is explained in \citep{Newman}, we only need to show that the rewriting system of braided diagrams is terminating and locally confluent.

\medskip

Firstly, we claim that the rewriting system of a braided diagram is terminating. Let $s$ be a state representing a braided diagram, and consider a path starting from it. This path cannot be infinite as:

\begin{itemize}
\item The number of $1$-moves and $2$-moves in the path is finite, as there is a finite number of splits and merges on $s$. Both $1$-moves and $2$-moves reduce them, whereas the other moves do not increase it. 
\item The number of $5$-moves and $6$-moves in the path is finite, as there is a finite number of crossings followed by splits, merges followed by crossings, merges followed by white vertices or white vertices followed by splits. The other moves do not increase these numbers.
\item Finally, the number of $3$-moves and $4$-moves in the path is finite, as there is a finite number of white vertices and white vertices followed by crossings. In this case, these numbers can be increased by performing $5$-moves and $6$-moves, but they can be performed a finite number of times, as we have seen. 
\end{itemize}   
Note that the previous facts are true because there are not oriented loops on any braided diagram, by definition. Therefore, the process of reducing is finite, so the rewriting system is terminating.

\medskip

On the other hand, it is easy to check the local confluence for all moves, that is, if we perform two different moves to the same braided diagram, there exists a braided diagram which is a reduction of both. For this purpose, it is enough to check all possibilities: given a fixed braided diagram $\Gamma$, let $\Gamma'$ (resp. $\Gamma''$) be the braided diagram obtained from $\Gamma$ by performing one $i$-move (resp. one $j$-move). One needs to prove that there exists a braided diagram $\Gamma'''$ which is a reduction of both $\Gamma'$ and $\Gamma''$, for all $i,j \in \{1, \dots, 6\}$. This is a laborious exercise left to the reader. 
\end{proof}

\begin{corollary}
\label{cor:bijection}
There is a bijection between classes of equivalent braided diagrams and classes of equivalent elements of $BV_n(H)$.
\end{corollary}

\begin{proof}

We only need to prove that there is a bijection between reduced braided diagrams and reduced elements of $BV_n(H)$, as performing reductions and expansions of an element in $BV_n(H)$ corresponds to performing $1$-moves and their inverses on the corresponding braided diagrams. 

\medskip

On the one hand, it is easy to see that a reduced element of $BV_n(H)$ produces a reduced braided diagram when doing the construction described at the beginning of this section. On the other hand, given a reduced braided diagram, we obtain the corresponding reduced element of $BV_n(H)$ by taking into account the following facts: every oriented path from a main source to a main sink has the following form:
$$v_{m_{so}} \rightarrow v_{s_1} \rightarrow \dots \rightarrow v_{s_i}   \rightarrow v^h \rightarrow v_{m_1} \rightarrow \dots \rightarrow v_{m_j} \rightarrow v_{m_{si}}$$
where $v_{m_{so}}$ (resp. $v_{m_{si}}$) is the main source (resp. the main sink), $v_{s_1}, \dots, v_{s_i}$ (resp. $v_{m_1}, \dots, v_{m_j}$) are splits (resp. merges) and $v^h$ is a white vertex. 
\begin{itemize}
\item If there were a merge followed by a split we could perform a $2$-move.
\item If there were a white vertex before a split or after a merge, then we could perform a $6$-move.
\item If there were more than one white vertex on an oriented path, then we could perform a $4$-move.
\item Regarding the crossings, note that all strands must cross after the set of splits and before the set of white vertices. Otherwise, we could perform a $5$-move and a $3$-move respectively.
\end{itemize}
Finally, consider making two cuts on every oriented path of the braided diagram as follows: one cut between the last split and before any crossing of the corresponding strand; and the second cut after the last crossing of the strand and before the white vertex (if exists). The result is a division of the braided diagram into three pieces. The first one is the domain tree $T$, the second one the braid $\beta$ and the third one the range tree $T'$ with a set of labels $\lambda$ on its leaves. That is the tree-pair representation of the element.\end{proof}

Note that this bijection turns out to be a homomorphism of groups: composing two elements of $BV_n(H)$ corresponds to composing their corresponding braided diagrams by appending an edge between the main sink of the former and the main source of the latter. Performing a $1$-move on a braided diagram corresponds to performing a reduction on the corresponding element of $BV_n(H)$. $2$-moves are used in order to compose braided diagrams, and they do not change the equivalence class of the composition, as the reader can check. The same argument can be applied to $3$-moves and $5$-moves. Finally, $4$-moves allow us to compose recursive braids, and $6$-moves do not properly appear on a composition of braided diagrams (they are compositions of inverses of $1$-moves with $2$-moves), but they are needed in order to obtain locally confluence on the rewriting systems.

\medskip

Therefore, we have that $\widehat{BV_n}(H)$ and $BV_{n,r}(H)$ are groups for every $n,r \geq 2$ and $H \leq \mathcal{B}_n$, as the definitions of moves and all proofs given in this section do not depend on the number of main sources or main sinks of the braided diagrams.

\section{Finite generation}\label{sec:fg}

In this section we show that the groups $BV_n(H)$ are finitely generated for every $n \geq 2$ and finitely generated $H \leq \mathcal{B}_n$. The proof is inspired by the one of \citep[Chapter~4]{Higman} for $V_n$, based on the depth of the elements; and it is different from the one used in \citep{Brin1,Brin2} for the braided Thompson group $BV_2$. 

\medskip

The forthcoming proofs can seem very technical at a first sight, but they are very intuitive when one draws the corresponding braided diagrams. We have pictured the diagrams with the details that we believe can be more difficult to understand by only reading. However, before going ahead, we advise to have in mind a clear picture of how the braided diagram of a tree-pair looks like.

\begin{definition}
We say that a tree $T$ has \emph{depth} $d \geq 1$ if it contains exactly $d$ different carets. The \emph{depth} of an element $v=(T,\beta,\lambda, T') \in {BV_n}(H)$ is the depth of both $T$ and $T'$.

\end{definition}

 For simplicity, we will sometimes refer to an element $v \in BV_n(H)$ as one of its tree-pair representations $(T,\beta, \lambda, T')$. We want to prove the following theorem. 

\begin{theorem}\label{theorem:fg}
Let $R$ be the $n$-ary tree of depth $1$ and let $R':=R[1]$ if $n=2$ and $R':=R[2]$ otherwise. Then $BV_n(\mathcal{B}_n)$, for $n>1$ can be generated using the $2n$ following elements:

\begin{itemize}

\item The $n$ generators of $F_n$ \citep{Brown}, that is, the elements $\left(R[n-1],Id,\overrightarrow{Id}, R[i]\right)$, for $i=0,\dots, n-2$ and $\left(R[n-1][(n-1)(n-1)], Id,\overrightarrow{Id}, R[n-1][(n-1)0]\right)$,

\item the element $\left(R',\sigma_{2n-2},\overrightarrow{Id} ,R'\right)$, and

\item the $n-1$ elements $\left(R,Id, \{\sigma_j, Id, \dots, Id\},R\right)$, for $j=1,\dots, n-1$.

\end{itemize}
\end{theorem}

\begin{figure}[h]
\labellist

\pinlabel $\sigma_1$ at 480 55
\pinlabel $\sigma_2$ at 586 55

\endlabellist
\centering
\includegraphics[scale=0.65]{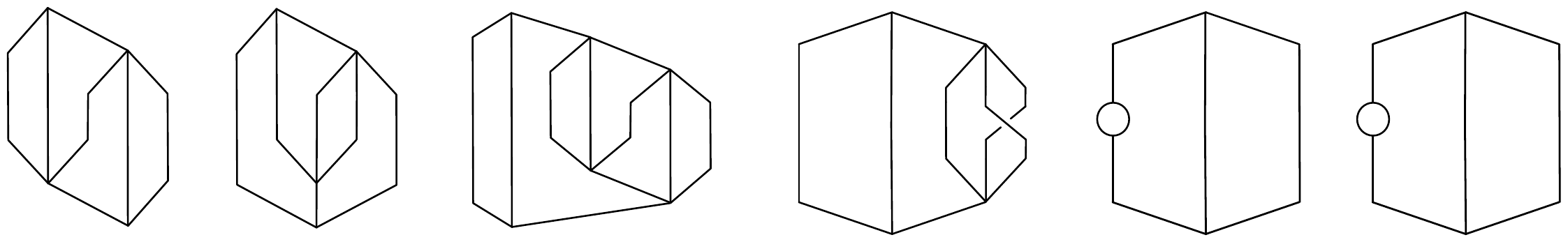}
\caption{The six generators of $BV_3(\mathcal{B}_3)$.}
\label{fig:generadores_teorema}
\end{figure}

\noindent
We depict in \autoref{fig:generadores_teorema} the six generators of $BV_3(\mathcal{B}_3)$. As we will see in Section 4.1.1, we can adapt this theorem for each $H \leq \mathcal{B}_n$.

\subsection{Proof of \autoref{theorem:fg}}

Our strategy uses a special type of braid, called ribbon. The following definition is a particular case of the ribbons defined in \citep{FRZ}. 
We say that a braid~$\beta$ is \emph{positive} if it can be written by using only positive powers of the Artin generators. Moreover, we say that a braid~$\beta$ is \emph{simple} if it is positive and every pair of strands crosses at most once. 

\begin{definition}
Consider $T$ and $T'$ to be two finite full $n$-ary trees with the same number~$l$ of leaves. Let $c$ and $c'$ be final carets of~$T$ and $T'$ respectively. Let $R$ be the set of all simple braids $\beta\in \mathcal{B}_l$ such that the natural injection in $\mathbb{R}^3$ of the braided diagram of $(T,\beta,T')$ contains a trivial sub-braid connecting the leaves of~$c$ to the leaves of~$c'$.
We define the \emph{ribbon $r$ connecting $c$ to $c'$} as the braid in $R$ having minimal length (as a word with respect to the positive Artin generators).

%
%
%
%
%
%

\end{definition}

\noindent
We include here \autoref{fig_ribbon2} in order to help the reader. A ribbon is called like that because the neighbourhood of the trivial sub-braid connecting the carets can be seen as a `ribbon' or a `tube' that the other strands cannot touch. This allows us to reduce a braided diagram associated to $(T,\beta,\overrightarrow{Id}, T')$ by performing a 1-move, as shown in \autoref{sec:3}. 

\begin{figure}[h]
\labellist
\pinlabel $c$ at 85 150
\pinlabel $c'$ at 188 40
\endlabellist
\centering
\includegraphics[scale=0.7]{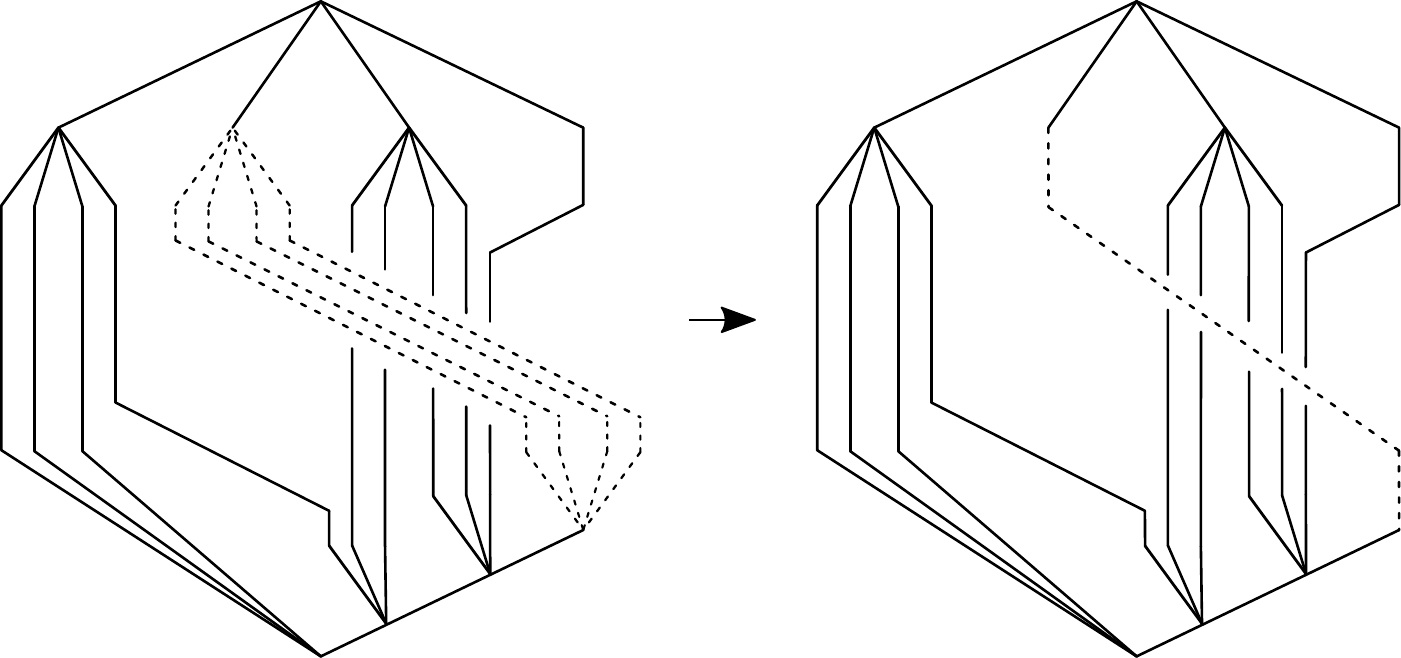}
\caption{A diagram containing a ribbon connecting $c$ to $c'$, and its reduction.}
\label{fig_ribbon2}
\end{figure}

\begin{definition} Let $T$ and~$T'$ be two full finite $n$-ary trees. The \emph{least common multiple} of~$T$ and~$T'$ is the minimal full finite $n$-ary tree~$T''$ with respect of expansions such that $T \subseteq T''$ and $T' \subseteq T''$ as rooted full finite $n$-ary trees. See \autoref{fig:leastcommonmultiple}.
\end{definition}

\begin{figure}[h]
\centering
\includegraphics[scale=1]{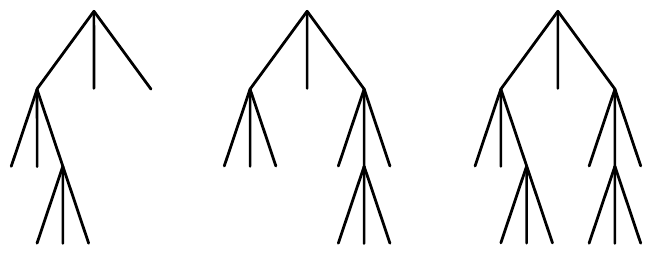}
\caption{The third tree is the least common multiple of the first two.}
\label{fig:leastcommonmultiple}
\end{figure}

Let~$R$ be the tree of depth~$1$. Define the \emph{`spine' tree $S_d$} of depth~$d$ as the full finite $n$-ary tree constructed inductively from~$R$ by expanding its last leaf, that is:
$$S_d = R[n-1][(n-1)(n-1)] \dots [\overbrace{(n-1) \dots (n-1)}^{d-1}].$$

\noindent
The proof of the following proposition gives a way to decompose elements of $BV_n(H)$.

\begin{proposition}\label{prop:fg}
Let $v=(T_1,\beta,\lambda, T_2)$ be an element of ${BV_n}(H)$ such that $T_1$ (and $T_2$) has depth $d > 4$. Then we can express $v$ as a product of elements in ${BV_n}(\mathcal{B}_n)$ having depth less than $d$.  Let $T_3$ be a tree of depth $d$ with exactly $3$ final carets. Then, each factor $(T'_1,\beta',\lambda',T'_2)$ satisfies one of the following properties:

\begin{itemize}

\item $T'_2=T_3$, $\beta'$ is a ribbon and $\lambda'= \overrightarrow{Id}$.

\item $T'_1=T_3$, $\beta'$ is a ribbon and $\lambda'=\overrightarrow{Id}$.

\item $T'_1=T'_2=T_3$, $\beta'=\sigma_i^{\pm 1}$, where $\sigma_i$ is an Artin generator and $\lambda'=\overrightarrow{Id}$.

\item $T'_1=T'_2$ has depth $1$, $\beta'$ is trivial and $\lambda'=\{\sigma_i^{\pm 1},Id,\dots, Id\}$, where $\sigma_i$ is an Artin generator.

\end{itemize}

\end{proposition}

\noindent
The moral of the proof, as well as in \citep{Higman}, is to decompose any element $v \in BV_n(H)$ of depth $d \geq 4$ into a product of elements which have less depth than $v$. We firstly do it for elements of $BV_n$, and then we deal recursive braids.

\begin{proof}
Let $c_1$ be a final caret of~$T_1$ and let $c_2$ be a final caret of~$T_2$. Let also $c_3$ and $c_4$ be two different final carets of $T_3$. Consider the ribbon~$r_1$ connecting $c_1$ to $c_3$ and  the ribbon~$r_2$ connecting~$c_4$ to~$c_2$. Then
\begin{align*}
\left(T_1,\beta, \overrightarrow{Id}, T_2\right) & \equiv \left(T_1, r_1 r_1^{-1} \beta r_2^{-1} r_2, \overrightarrow{Id}, T_2\right)\\ & \equiv \left(T_1,r_1,\overrightarrow{Id}, T_3\right) \left(T_3,r_1^{-1}\beta r_2^{-1},\overrightarrow{Id}, T_3\right)\left(T_3,r_2, \overrightarrow{Id}, T_2\right),
\end{align*}
where $\equiv$ means that all tuples (or compositions of tuples) represent the same element of $BV_n(H)$.
The diagram constructed from $(T_1,r_1,\overrightarrow{Id}, T_3)$ is not reduced, because there is a ribbon $r_1$ connecting $c_1$ to $c_3$ and one can perform a $1$-move. Hence, the corresponding element has depth less than $d$. Using the same argument, the element represented by  $(T_3,Id,r_2, T_2)$ has depth less than $d$. Also notice that $(T_3,r_1^{-1}\beta r_2^{-1},\overrightarrow{Id}, T_3)$ is equivalent to a product of elements of the form $(T_3, \sigma_i^{\pm 1},\overrightarrow{Id}, T_3)$.  Since~$T_3$ has 3 final carets and $\sigma_i^{\pm 1}$ is a crossing of two consecutive strands, we can always reduce its diagram. This means that it has depth less than~$d$.

\medskip
\noindent
Let $T_4$ be the spine tree~$S_d$. Let also~$T_5$ be the full finite $n$-ary tree of depth~$1$ and notice that the least common multiple of $T_4$ and~$T_5$ is~$T_4$, as $T_5 \subset T_4$. Suppose that we have $(T_1,\beta,\lambda_1, T_2)$, where $\lambda_1 = \{\ell_1,\dots,\ell_j,\dots,\ell_m\}$. We show that is possible to obtain an element $(T_1,\beta,\lambda'_1, T_2)$ from the previous one such that $\lambda'_1 = \{\ell_1,\dots,\ell_j\sigma_i^{\pm 1},\dots,\ell_m\}$ for some $\sigma_i$. Let $\gamma$ be any braid performing the permutation $\gamma(j)=1$, and $\lambda_{i,j}=\{\ell_k\}_{k=1}^m$ such that $\ell_j=\sigma_i^{\pm 1}$ and $\ell_k=Id$ if $k\neq j$.  Then, we have
\begin{align*}
\left(T_1,\beta,\lambda_1, T_2\right) \left(T_2, \gamma, \overrightarrow{Id},  T_4\right) \left(T_5, Id, \lambda_{i,1}, T_5\right) \left(T_4, \gamma^{-1}, \overrightarrow{Id}, T_2 \right) & \equiv \left(T_1,\beta,\lambda_1, T_2\right) \left(T_2, Id, \lambda_{i,j}, T_2 \right) \\ & \equiv \left(T_1,\beta,\lambda_2, T_2\right). 
\end{align*}
Since we have already shown that we can express $(T_2, \gamma^{\pm 1}, \overrightarrow{Id}, T_4)$ using the desired generators, this finishes the proof.
\end{proof}

\medskip
\noindent
Note that the previous proof works for depth $d \geq 5$ because $5$ is the minimal depth that $T_3$ can have for $n =2$. Also notice that the previous proposition can be refined to obtain a set of generators for $BV_n(H)$ if $H$ is finitely generated, by using recursive braids (labels) that correspond to the generators of $H$.

\medskip

Let $T$ be the depth $1$ tree. We define

$$T(n) = \begin{cases}
T[0][1][00][01], &  \mbox{if} \ n = 2, \\
T[0][1][2], & \mbox{if} \ n > 2. \\
\end{cases}$$

Let $m(n)$ be the number of leaves of~$T(n)$.

\begin{corollary}\label{cor:fg}
For every $n \geq 2$, $BV_n(\mathcal{B}_n)$ is finitely generated. In addition, every generator has depth at most 5 (4 if $n\geq 3$), and has a tree pair representative $(T,\beta, \lambda, T')$ satisfying one of the following conditions.

\begin{itemize}

\item $T'=T(n)$, $\beta=Id$ and $\lambda=\overrightarrow{Id}$. This generator is denoted by $e_T$, regarding the domain tree $T$ of $(T,\beta, \lambda, T')$.

\item $T=T'=T(n)$, $\beta=\sigma_i$, where $\sigma_i$ is an Artin generator of $\mathcal{B}_{m(n)}$ and $\lambda=\overrightarrow{Id}$. This generator is denoted by $h_i$.

\item $T=T'$ has depth $1$, $\beta$ is trivial and $\lambda=\{\sigma_j, Id,\dots, Id\}$, where $\sigma_j$ is an Artin generator of~$\mathcal{B}_n$. This generator is denoted by $g_j$.

\end{itemize}

\end{corollary}

\begin{proof}

As one can always perform the inverse of a $1$-move to a diagram, we can always apply \autoref{prop:fg} (performing the inverse of a $1$-move corresponds to expanding the domain and the range trees of the corresponding tree pair). Then, by applying induction on the depth of elements, it follows that we can use as generators the elements described in \autoref{prop:fg} by replacing $T_3$ with $T(n)$. All these elements have depth less than~4 (or 3 if $n\geq 3$). 
\medskip

\noindent
Notice that $(T(n),\sigma_i,\overrightarrow{Id},T(n))$ is the inverse of $(T(n),\sigma_i^{-1},\overrightarrow{Id},T(n))$, so we can consider only the first one as a generator. Analogously, we can discard the $(T,Id,\lambda,T)$ containing $\lambda=\{\sigma_j^{-1},Id,\dots, Id\}$. On the other hand, if $\beta$ is a ribbon, we have that \[\left(T,\beta, \overrightarrow{Id}, T(n) \right)\equiv \left(T,Id, \overrightarrow{Id}, T(n)\right) \left(T(n), \beta,\overrightarrow{Id}, T(n)\right).\]
Observe that $(T(n), \beta,\overrightarrow{Id}, T(n))$ can be written as a product of elements with representatives $(T(n),\sigma_i^{\pm 1},\overrightarrow{Id},T(n))$. Then use $(T,Id,\overrightarrow{Id}, T(n))$ as a generator and get rid of $(T,\beta,\overrightarrow{Id}, T(n) )$.
Similarly, we can replace the generators of the form $(T(n),\beta,\overrightarrow{Id}, T)$ by $(T(n),Id,\overrightarrow{Id}, T)$. Finally, notice that $(T(n),Id,\overrightarrow{Id}, T)$ is the inverse of $(T,Id,\overrightarrow{Id}, T(n))$, hence we can discard it.
\end{proof}

\begin{remark}\label{remark:gen}
Notice that the set of generators $\{ e_T, 
h_1, \dots, h_{m(n)-1} \}$, for all trees $T$ of depth less than $5$ (or $4$ if $n>2$), generate all elements in $BV_n$. Also, if $S_H$ is a generating set for $H\leq \mathcal{B}_n$ and $T$ is the $n$-ary tree of depth 1, we can substitute the set of generators $g_j$ by the set of generators $(T,Id, \{s, Id,\dots ,Id \}, T)$ for every $s \in S_H$.
\end{remark}

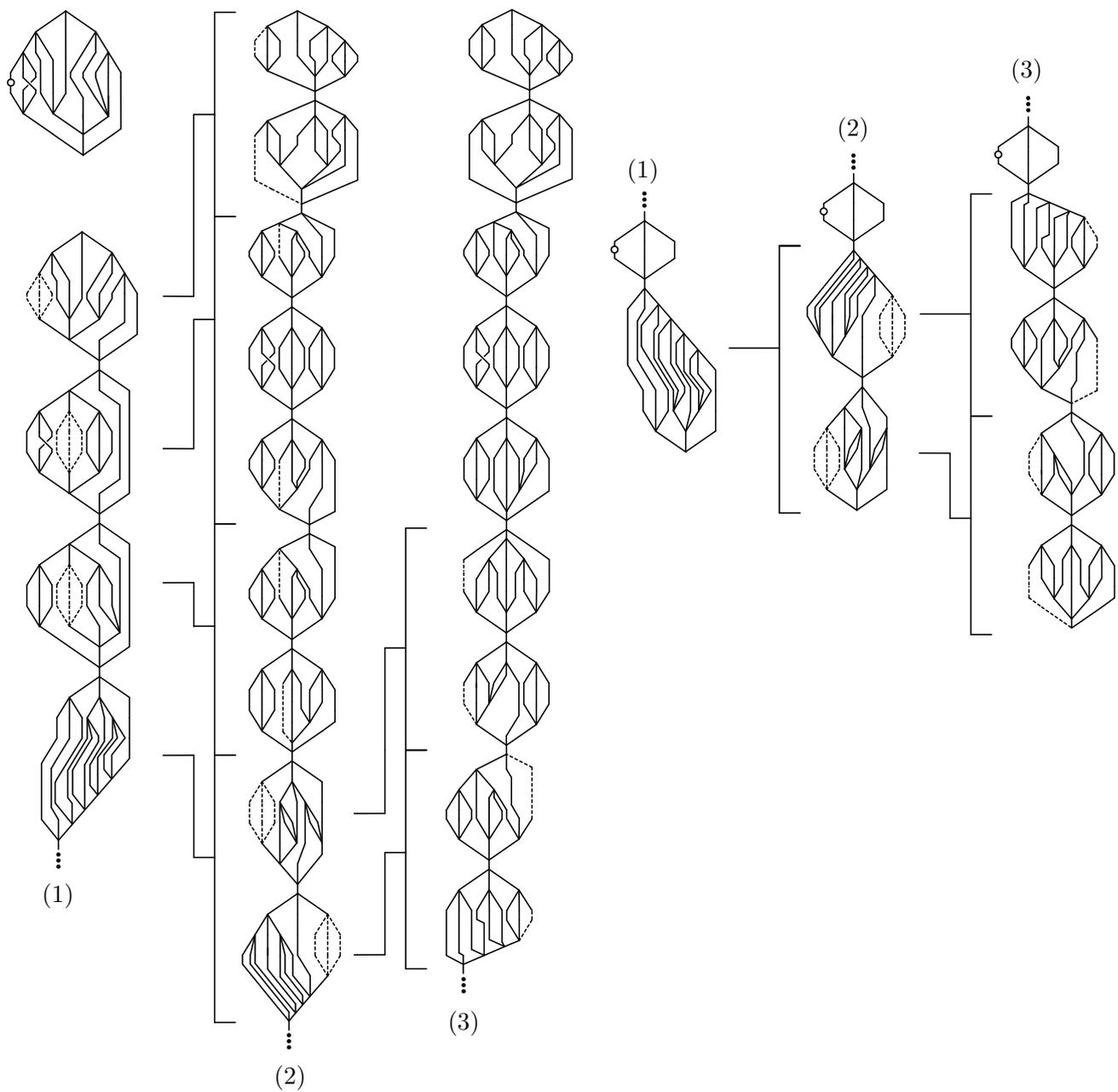
\begin{figure}[p!]
\centering
\labellist
\pinlabel $(1)$ at 28 100
\pinlabel $(2)$ at 155 0
\pinlabel $(3)$ at 250 30

\pinlabel $(1)$ at 348 500
\pinlabel $(2)$ at 463 523
\pinlabel $(3)$ at 558 555
\endlabellist
\includegraphics[scale=0.8]{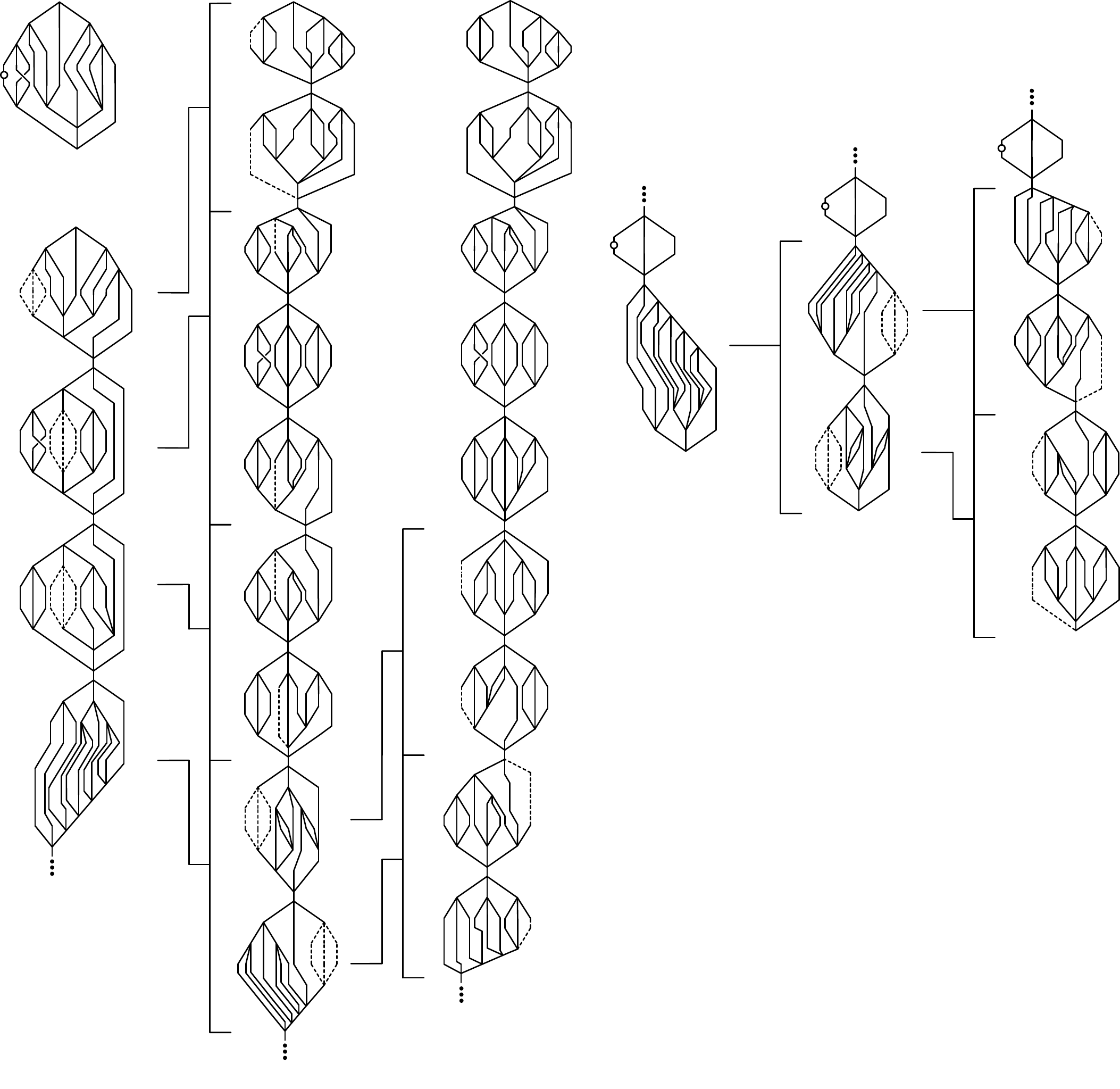}
\caption{How to decompose $v\in BV_3(\mathcal{B}_3)$ in generators of depth less than 5.}
\label{decomposition}
\end{figure}

\medskip

\noindent
\textbf{Example.} \autoref{decomposition} is a visual example of how to decompose an element  of $BV_3(\mathcal{B}_3)$, represented by $(T,\sigma_2, \{\sigma_i, Id, \dots, Id\}, T')$, using the generators of \autoref{cor:fg}. Let $T_d(3)$ be a $n$-ary tree of depth $d$ with exactly $3$ final carets. The reader will notice that, to make this example simpler, we always choose $T_d(3)$ such that we already have the necessary (trivial) ribbons to reduce the depth of elements. We describe now the decomposition process.

\medskip\noindent
The element $(T,\sigma_2, \{\sigma_i, Id, \dots, Id\}, T')$ has depth 5 and is decomposed as $$\left(T, Id,\overrightarrow{Id}, T_5(3)\right) \left(T_5(3), \sigma_2,\overrightarrow{Id}, T_5(3)\right) \left(T_5(3), Id,\overrightarrow{Id}, T'\right)  \left(T', Id,\overrightarrow{Id}, S_5 \right) g_i \left( S_5, Id,\overrightarrow{Id}, T' \right)$$

\begin{enumerate}

\item The element $(T, Id,\overrightarrow{Id}, T_5(3))$ has a ribbon connecting the leftmost final carets of $T$ and $T_5(3)$. Then, there is an equivalent representative $(T_1, Id,\overrightarrow{Id}, T_1')$ of depth $4$ obtained by performing a 1-move that removes the previous final carets. This new representative is decomposed as the product $(T_1,Id,\overrightarrow{Id},T(3))(T(3),Id,\overrightarrow{Id}, T_1')$, where $T_1$ (resp. $T'_1$ is a reduction of $T$ (resp. $T_5(3)$).

\item  The element $(T_5(3), \sigma_2,\overrightarrow{Id}, T_5(3))$ has a ribbon connecting the second final carets of both $T_5(3)$'s. Then, there is an equivalent representative $(T_2, \sigma_2,\overrightarrow{Id}, T_2)$  of depth $4$ obtained by removing the previous final carets. This new representative is decomposed as the product $$(T_2,Id,\overrightarrow{Id},T(3))(T(3),\sigma_2,\overrightarrow{Id},T(3))(T(3),Id,\overrightarrow{Id}, T_2).$$

\item The element $(T_5(3), Id,\overrightarrow{Id}, T')$ has a ribbon connecting the second final carets of $T'$ and $T_5(3)$. Then there is an equivalent representative $(T_3, Id,\overrightarrow{Id}, T_3')$  of depth $4$ obtained by removing the previous final carets. This new representative is decomposed as the product $(T_3,Id,\overrightarrow{Id},T(3)) (T(3),Id,\overrightarrow{Id}, T_3')$.

\item The element $(T', Id,\overrightarrow{Id}, S_5)$ is decomposed as $(T', Id,\overrightarrow{Id}, T_5(3)) (T_5(3), Id,\overrightarrow{Id}, S_5)$. In this case, one can redefine $T_5(3)$ such that:

\begin{enumerate}

\item The element $(T', Id,\overrightarrow{Id}, T_5(3))$ has a ribbon connecting the leftmost final carets of $T'$ and $T_5(3)$. Then there is an equivalent representative $(T_4, Id,\overrightarrow{Id}, T_4')$  of depth $4$ obtained by removing the previous final carets. This new representative is decomposed as  the product $(T_4,Id,\overrightarrow{Id},T(3)) (T(3),Id,\overrightarrow{Id}, T_4').$

\item $(T_5(3), Id,\overrightarrow{Id}, S_5)$  has a ribbon connecting the rightmost final carets of $T'$ and $T_5(3)$. Then there is an equivalent representative $(T_5, Id,\overrightarrow{Id}, T_5')$  of depth $4$ obtained by removing the previous final carets. This new representative is decomposed as  the product $(T_5,Id,\overrightarrow{Id},T(3)) (T(3),Id,\overrightarrow{Id}, T_5').$ 

\end{enumerate} 

Therefore, the element $(T,\sigma_2, \{\sigma_i, Id, \dots, Id\}, T')$ is also represented by the word $$e_{T_1} e^{-1}_{T_1'} e_{T_2} h_2 e^{-1}_{T_2} e_{T_3} e^{-1}_{T_3'}  e_{T_4} e^{-1}_{T_4'} e_{T_5} e^{-1}_{T_5'}  g_i  e_{T_5'} e^{-1}_{T_5} e_{T_4'} e^{-1}_{T_4}.$$
\end{enumerate}

\bigskip\noindent
The set of generators in \autoref{cor:fg} can be further reduced. The following two lemmas will prove that we just need only one generator of type $h_i$, namely $h_{m(n)-1}$.

\begin{lemma}\label{lemma1:fg}
The generators $h_i$, for $1 \leq i \leq n-1$, of \autoref{cor:fg} can be expressed with words containing only generators of type $e$ and $g$. 
\end{lemma}

\noindent
Note that, once this lemma will be proved, the generators $h_i$ for $i \in \{n,\dots,m(n)-1\}$ will remain. They will be treated in the next lemma, proving that only $h_{m(n)-1}$ is needed.

\begin{proof}

First of all, notice that $v_i:=(T(n), \sigma_i, \lambda_i, T(n))$, where $\lambda_i=  \{\overbrace{\sigma_i,\dots, \sigma_i}^{n},Id,\dots, Id\}$ is equivalent to the  generator $g_i$. 

\medskip\noindent
Define \begin{align*}
w_i^{(k)} & :=\left(T(n),Id, \{\mu_j\}_{j=0}^{m(n)-1}, T(n)\right) \ \mbox{where} \ \begin{cases}  \mu_j= \sigma_i, \ \mbox{if} \ j =k,\\ \mu_j = Id, \ \mbox{otherwise}. \\
\end{cases}\\
z_i & :=\left(T(n),\sigma_i, \{\xi_j\}_{j=0}^{m(n)-1}, T(n)\right) \ \mbox{where} \ \begin{cases}
\xi_j =\sigma_i, \  \mbox{if} \ j \in \{i+1,\dots, n-1\},\\
\xi_j=Id, \ \mbox{otherwise}.\\
\end{cases}
\\
\end{align*} for $i \in \{1,\dots, n-1\}$.

\bigskip\noindent
\emph{Claim.} For any $i\in \{1,\dots, n-1\}$, the elements $z_i$ and $w_i^{(k)}$, with $0\leq k \leq i$, can be generated by using only type $e$ and type $g$ generators. 

\medskip\noindent
We prove it by induction. For $k=0$, ${w_i^{(0)}}\equiv e_{S_d}^{-1} g_i e_{S_d}$, where $S_d$ is a spine tree and $d$ is the depth of $T(n)$ depending on whether $n =2$ or not. Suppose that the claim holds for $j<k<i$. We need to prove that we can generate $w_i^{(k)}$. Denote by $r_p$ the braid $\sigma_{p}\sigma_{p-1}\cdots \sigma_1$ for $p>1$ and consider the element represented by $r(k):= z_{k} z_{k-1}\cdots z_{1}$ for $0\leq k < i$. A representative of this element is $(T(n), r_k, \rho_k, T(n))$, where $$\rho_k=\{Id, Id,r_1,r_2,\dots, r_{k},\overbrace{r_{k+1},\dots, r_{k+1}}^{n-k-2}, Id,\dots, Id\}.$$ Then we have that  $w_i^{(k)}\equiv r(k) e_S^{-1} g_i e_S r(k)^{-1}$, for $0\leq k<i$ (see \autoref{fig:lema2}).

\noindent
Now define $x_i:= v_i \left( w_i^{(i-1)}\right)^{-1}$, so $w_i^{(i)}\equiv v_i  x_i ^{-1}$ (see \autoref{fig:lema1}). To finish the proof of the claim, notice that $z_i\equiv v_i  \left(w_i^{(0)}\right)^{-1} \left(w_i^{(1)}\right)^{-1}\cdots \left(w_i^{(i)}\right)^{-1}$.

\medskip

To prove the statement of the lemma, we also proceed by induction. For $i=n-1$, $h_{n-1}=z_{n-1}$. Now suppose that the statement holds for $i+1$, we prove it for $i$. By Claim~1, we have seen that $w_i^{(i-1)}$ can be expressed using generators of type $e$ and $g$. We prove that $h_i$ is equal to the following product: \[ h_{i+1} w_i^{(i-1)} z_i  h_{i+1}  z_{i}^{-1}  \left( w_i^{(i-1)}\right)^{-1}  h_{i+1}^{-1}.\]

\begin{figure}[H]
\centering
\labellist

\pinlabel $\sigma_2$ at 288 332 
\pinlabel $\sigma_1$ at 269 332 
\pinlabel $\sigma_i$ at 250 332 

\pinlabel $\sigma_2^{-1}$ at 290 308 
\pinlabel $\sigma_1^{-1}$ at 269 308 

\pinlabel $w_i^{(2)}$ at 347 193 
\pinlabel $\equiv$ at 225 193 
\pinlabel $\equiv$ at 100 193 

\pinlabel $r(2)$ at -11 330 
\pinlabel $e_S^{-1}$ at -11 250
\pinlabel $g_i$ at -11 185 
\pinlabel $e_S$ at -11 125 
\pinlabel $r(2)^{-1}$ at -16 40 

\endlabellist
\includegraphics[scale=0.9]{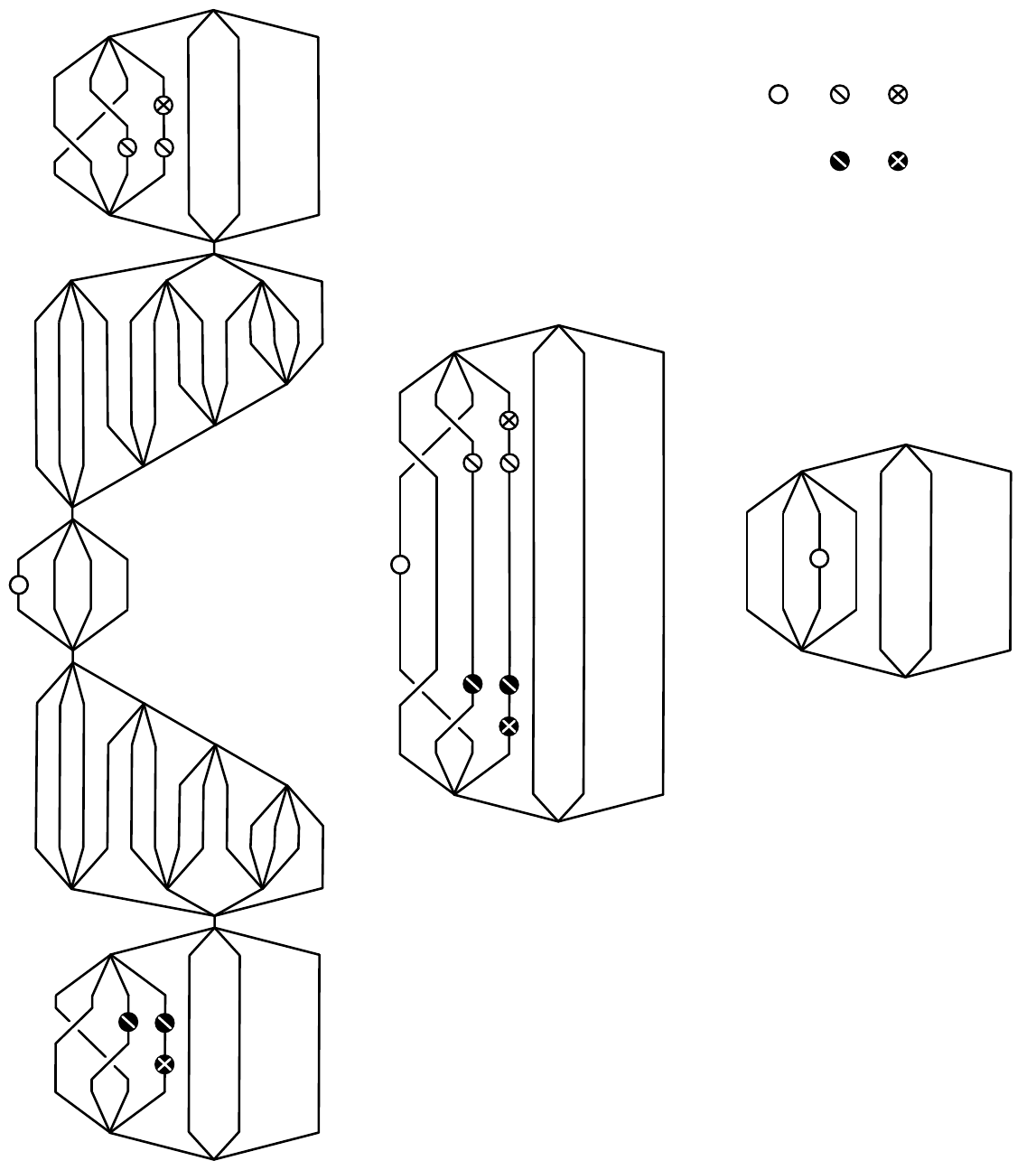}
\caption{$w_i^{(2)}\equiv r(2)e_S^{-1}  g_i  e_S r(2)^{-1}$ for $n=4$. Consider the diagrams up to equivalence to match the labels.}
\label{fig:lema2}
\end{figure}
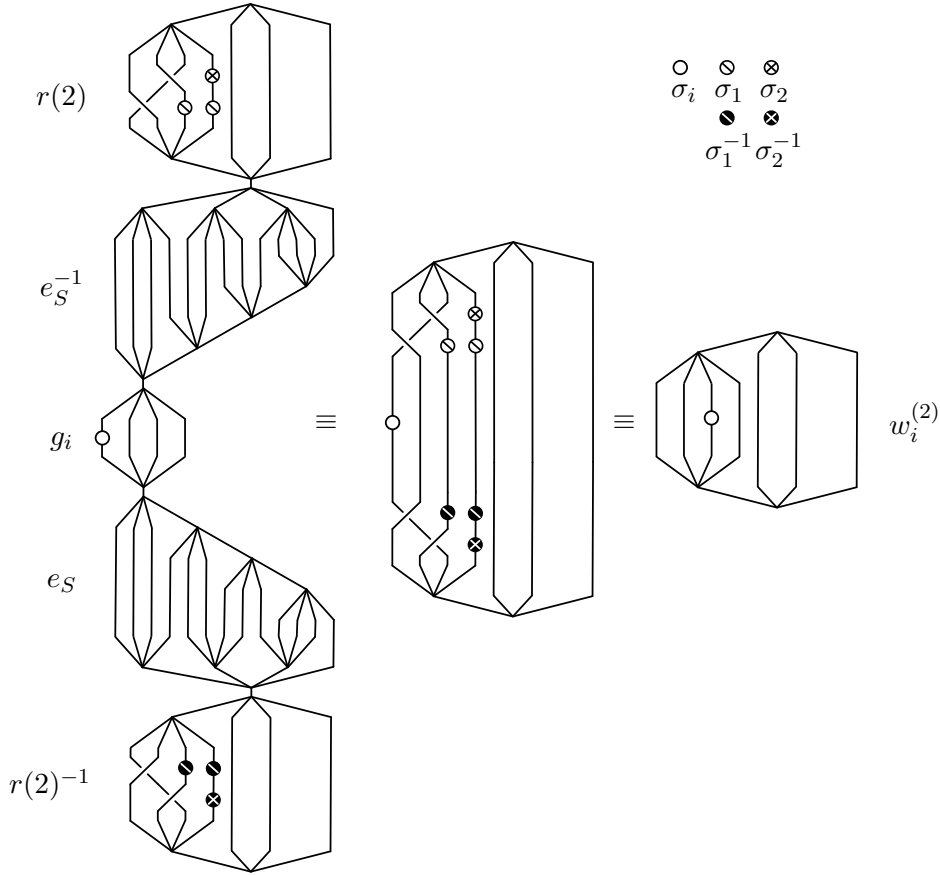
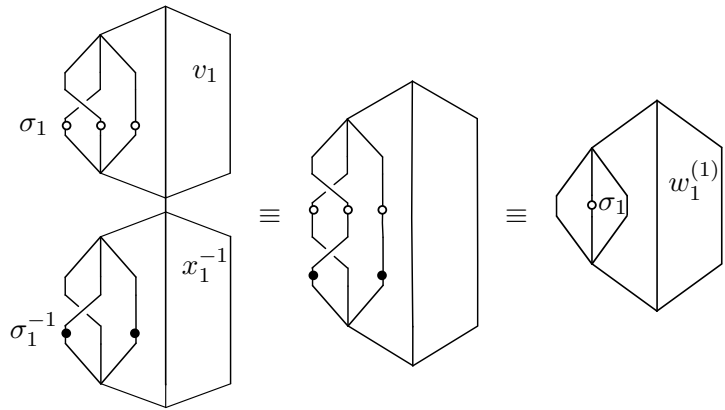
\begin{figure}[h]
\centering
\labellist
\pinlabel $\sigma_1^{-1}$ at -10 40 
\pinlabel $\sigma_1$ at -10 135 
\pinlabel $v_1$ at 70 160
\pinlabel $x_1^{-1}$ at 70 70
\pinlabel $\equiv$ at 100 93
\pinlabel $\equiv$ at 215 93
\pinlabel $\sigma_1$ at 260 98
\pinlabel $w_1^{(1)}$ at 298 110

\endlabellist
\includegraphics[scale=0.8]{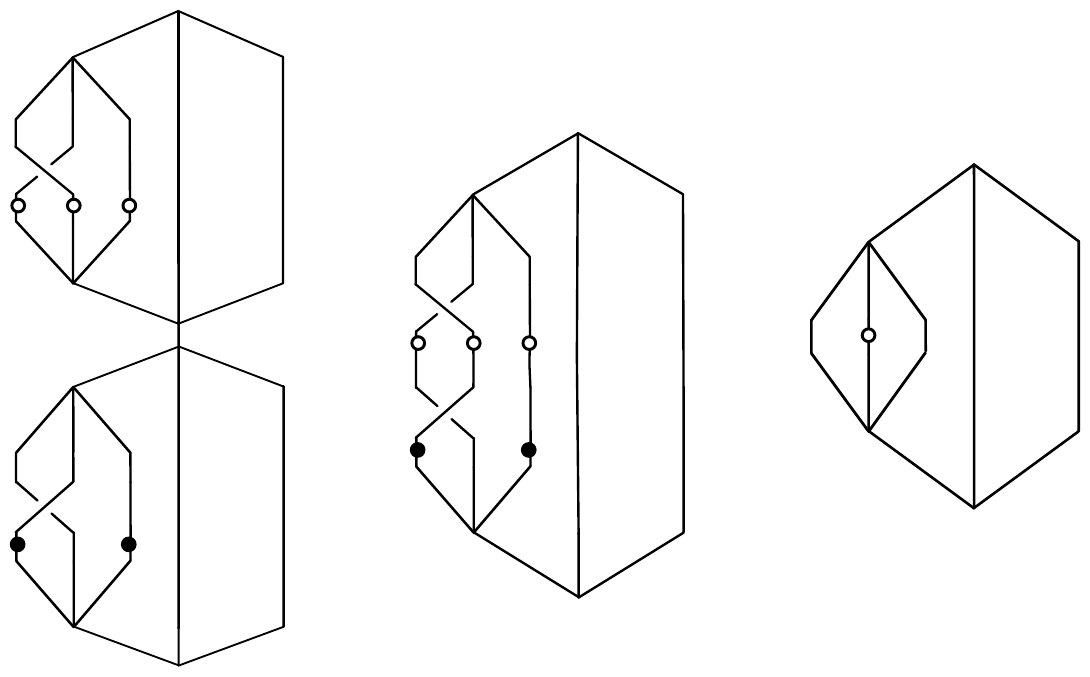}
\caption{$w_1^{(1)}\equiv v_1 x_1 ^{-1}$ for $n=3$. Consider the diagrams up to equivalence to match the labels.}
\label{fig:lema1}
\end{figure}
Notice that $ z_i  h_{i+1}  z_{i}^{-1}\equiv \left(T(n),\sigma_i\sigma_{i+1}\sigma_{i}^{-1}, \{\nu_j\}_{j=0}^{m(n)-1} ,T(n)  \right)$, where $\nu_{i-1}= \sigma_i$, $\nu_{i+1}=\sigma_i^{-1}$ and $\nu_j=Id$ if $j\neq i,i+2$ (see \autoref{fig:lema3}).  
By the braid relation $\sigma_i\sigma_{i+1}\sigma_{i}^{-1} (i)= i+2$, the conjugate of $\left(T(n),\sigma_i\sigma_{i+1}\sigma_{i}^{-1}, \{\nu_j\}_{j=0}^{m(n)-1} ,T(n)\right)$ by $w_i^{(i-1)}$ has trivial labels. Hence we obtain 
\[ h_{i+1}\left(T(n),\sigma_i\sigma_{i+1}\sigma_{i}^{-1}, \overrightarrow{Id} ,T(n)\right)   h_{i+1}^{-1}  \equiv \left(T(n),\sigma_{i+1}\sigma_i\sigma_{i+1}\sigma_{i}^{-1} \sigma_{i+1}^{-1}, \overrightarrow{Id} ,T(n)  \right).\]
By using braid relations we know that $\sigma_{i+1}\sigma_i\sigma_{i+1}\sigma_{i}^{-1} \sigma_{i+1}^{-1}=\sigma_i$, so the expression is equivalent to $h_i$, as we wanted to prove.
\end{proof}

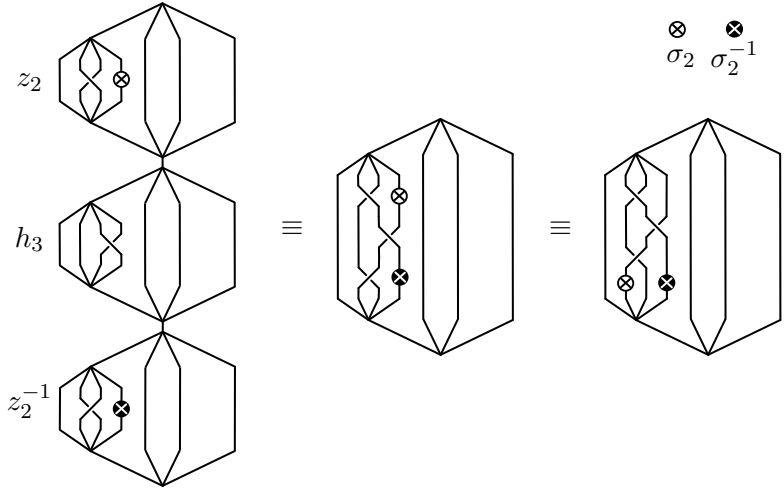
\begin{figure}[H]
\centering
\labellist

\pinlabel $\sigma_2$ at 235 164
\pinlabel $\sigma_2^{-1}$ at 255 164
\pinlabel $z_2$ at -8 155
\pinlabel $h_3$ at -8 96
\pinlabel $z_2^{-1}$ at -8 35
\pinlabel $\equiv$ at 190 100
\pinlabel $\equiv$ at 90 100

\endlabellist
\includegraphics[scale=1]{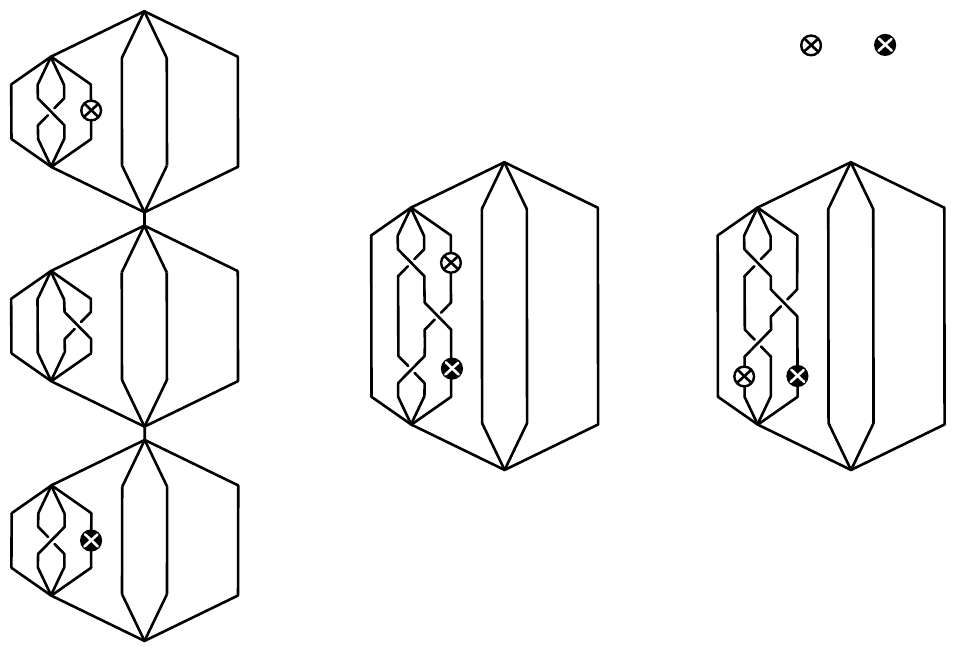}
\caption{ $ z_2  h_{3} z_{2}^{-1}\equiv (T(4),\sigma_2\sigma_{3}\sigma_{2}^{-1}, \nu_2 ,T(4))$ for $n=4$. Consider the diagrams up to equivalence to match the labels.}
\label{fig:lema3}
\end{figure}

\begin{lemma}\label{lemma2:fg}
All generators $h_i$, for $i\geq n$, can be expressed in terms of generators $h_i$ for $i<n$ and $h_{m(n)-1}$.
\end{lemma}

\begin{proof} We proceed by induction. Let $c$ be the right-most final caret of $T(n)$:
$$c= \left\{ \begin{array}{ll}
\{10,11\} &  \text{ if } n=2, \\
\{20,\dots,2(n-1)\} & \text{ otherwise. }
\end{array} \right.$$
Let also
$$T_1= \left\{ \begin{array}{ll}
T(n)[c]^{-1}[000] &  \text{ if } n=2, \\
T(n)[c]^{-1}[00] & \text{ otherwise; }
\end{array} \right. \text{ and } \quad T_2= \left\{ \begin{array}{ll}
T(n)[c]^{-1}[001] &  \text{ if } n=2, \\
T(n)[c]^{-1}[01] & \text{ otherwise. }
\end{array} \right.$$
Note that $h_1 \equiv \left(T_1, \sigma_{n}\sigma_{n-1}\dots \sigma_1  ,\overrightarrow{Id}, T_2\right).$ Hence we have the product (see \autoref{fig:lema4}): \[h_n\equiv  \left(T(n),Id, \overrightarrow{Id}, T_1\right) h_1  \left(T_2,Id, \overrightarrow{Id}, T(n)\right)  h_1^{-1}  h_2^{-1} \cdots  h_{n-1}^{-1} .\]

Suppose that the statement of the lemma is true for every $j<i$ when $i>n$, and let $T'(n):= T(n)[c]^{-1}[n-1]$ for $n>3$, that is, the tree obtained from $T(n)$ by attaching a final caret to its last leaf and erasing the final caret $c$. If $n=2,3$, we set $T'(n)=T(n)$. Also let $T''(n):= T'(n)[\{(n-1)0,\dots,(n-1)(n-1)\}]^{-1}[0]$.  Notice that, for $1<i<m(n)-n$, we have \[\left(T'(n),\sigma_i ,\overrightarrow{Id},T'(n)\right)\equiv \left(T''(n), \sigma_{i+n-1},\overrightarrow{Id}, T''(n)\right)\]
Then we obtain $h_i$ using the following product (see \autoref{fig:lema5}): \[h_i\equiv  \left(T(n),Id, \overrightarrow{Id }, T''(n) \right) \left(T'(n),Id, \overrightarrow{Id }, T(n) \right) h_{i-n+1}  \left(T(n),Id, \overrightarrow{Id }, T'(n) \right) \left(T''(n),Id, \overrightarrow{Id }, T(n) \right).   \]
\end{proof}

\begin{figure}[H]
\centering
\labellist

\pinlabel $h_2^{-1}$ at -13 50
\pinlabel $h_1^{-1}$ at -13 145
\pinlabel $h_1$ at -13 325
\pinlabel $(T(3),Id,T_1)$ at -45 420
\pinlabel $(T_2,Id,T(3))$ at -45 235
\pinlabel $h_1$ at 505 230
\pinlabel $\equiv$ at 376 230

\endlabellist
\includegraphics[scale=0.65]{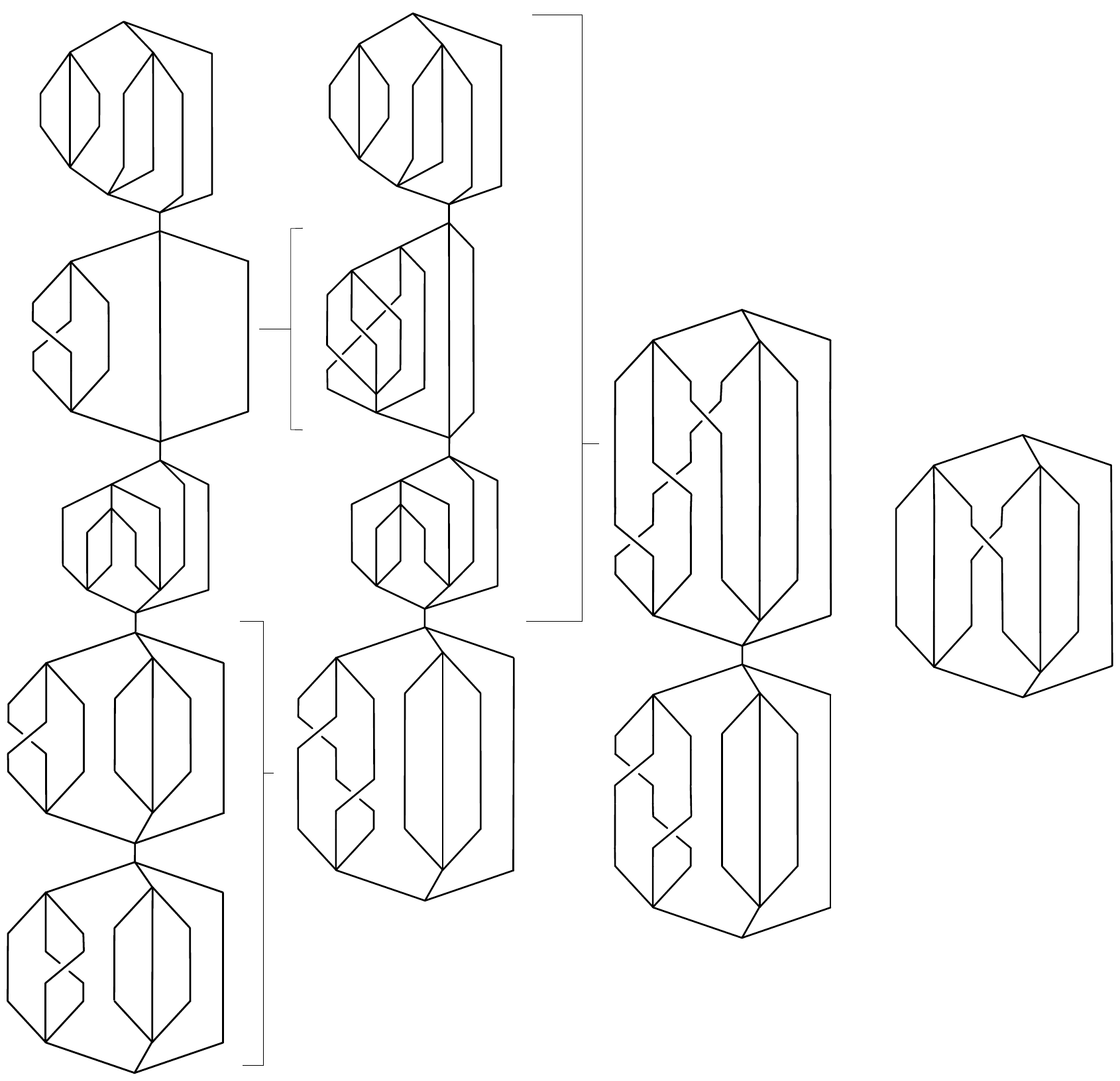}
\caption{ $h_3\equiv (T(3),Id, T_1) h_1  (T_2,Id, T(n)) h_1^{-1}  h_2^{-1}$ for $n=3$.}
\label{fig:lema4}
\end{figure}
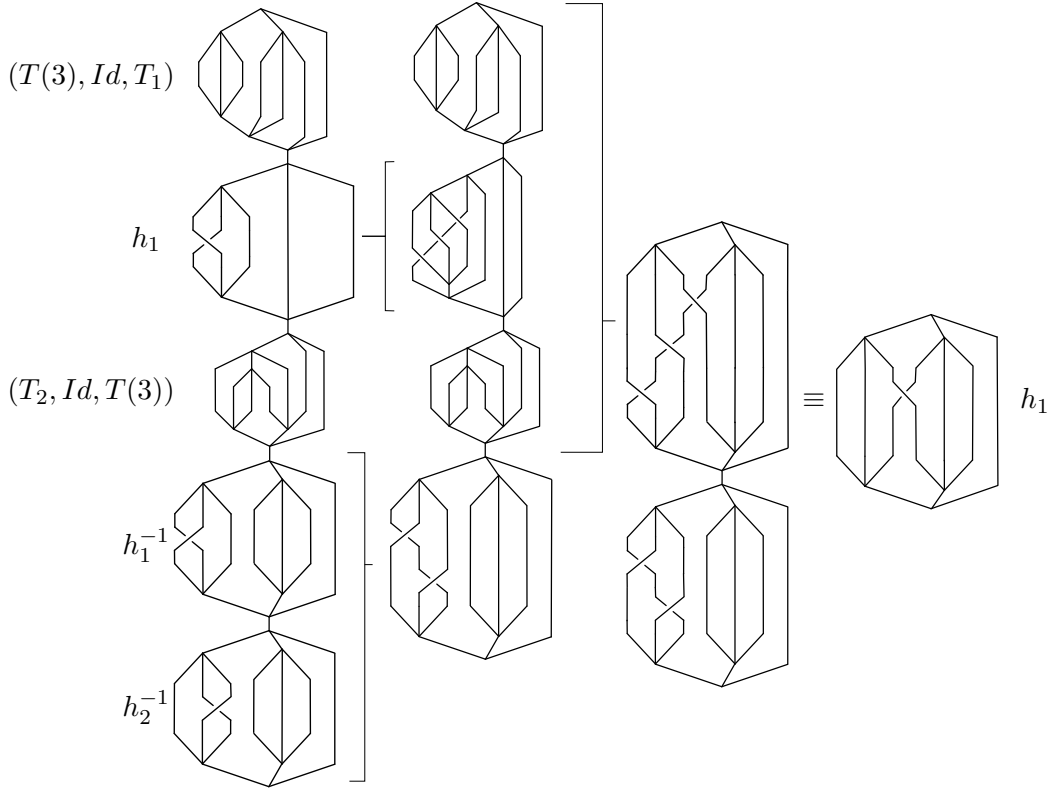

\bigskip

\begin{figure}[H]
\centering
\labellist
\pinlabel $\equiv$ at 65 75
\pinlabel $h_{11}$ at 142 75
\pinlabel $(T'(4),\sigma_8,T'(4))$ at -38 78
\pinlabel $(T(4),Id,T''(4))$ at -38 125
\pinlabel $(T''(4),Id,T(4))$ at -38 25

\endlabellist
\includegraphics[scale=1.2]{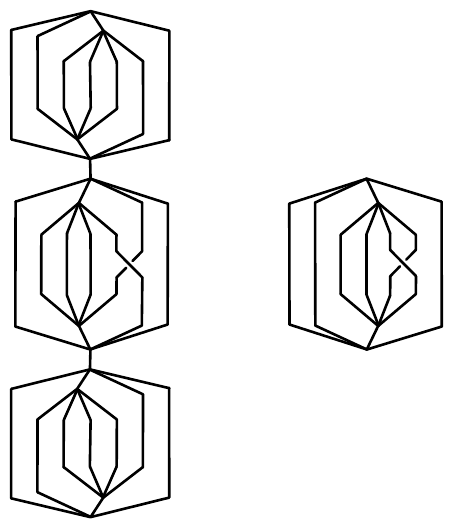}
\caption{How to shift a crossing to the right using conjugation when $n=4$. Consider the diagrams up to equivalence to match the labels.}
\label{fig:lema5}
\end{figure}

\noindent
The generators of type $e$ belong to the well studied Thompson's group $F_n$, consisting of all elements represented by triples $(T, \tau, T')$ where $\tau$ is trivial. This group is finitely generated:

\begin{lemma}[{\citealp{CFP, Brown}}]\label{lemma3:fg}
The generators of type $e$ in \autoref{cor:fg} can be generated by using the $n$ elements $x_0,\dots,x_{n-1}$ depicted in \autoref{fig:generators}. 
\end{lemma}

\noindent
Therefore, we have proved that the $2n$ elements in \autoref{theorem:fg}, which are precisely $x_0,\dots,x_{n-1}$, $h_{m(n)-1}$, $g_1,\dots, g_{n-1}$, generate $BV_n(\mathcal{B}_n)$.

\begin{figure}[h]
\centering
\labellist

\pinlabel $x_i$ at 55 53
\pinlabel $i$ at 96 51
\pinlabel $i\in\{0,\dots,n-2\}$ at 63 11
\pinlabel $x_{n-1}$ at 206 53

\endlabellist
\includegraphics[scale=1.5]{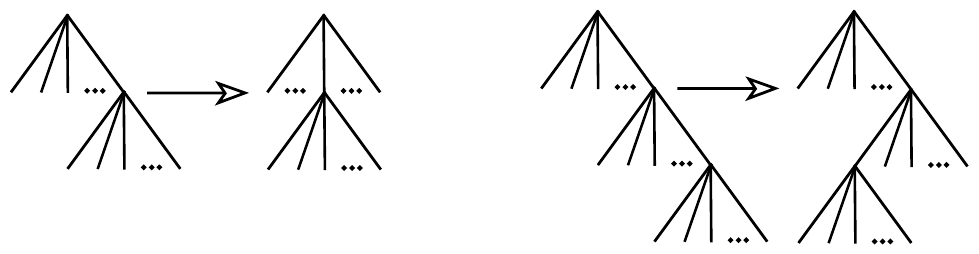}
\caption{Generators of $F_n$.}
\label{fig:generators}
\end{figure}
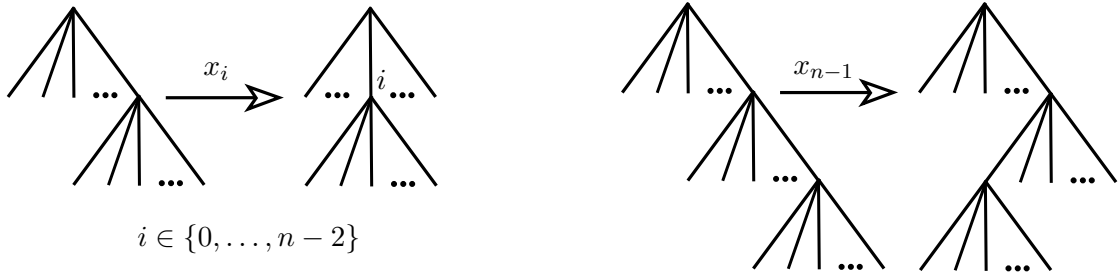

\subsubsection{Other generating systems for specific subgroups $H$}

The reader must have noticed that in most of the cases, the set of generators given in \autoref{theorem:fg} is not contained in $BV_n(H)$. 

\medskip
\noindent
If $H$ is trivial,
\autoref{remark:gen} together with \autoref{lemma2:fg} and \autoref{lemma3:fg} provide another set of $2n$ generators for $BV_n$ that is contained in $BV_n$. These elements are $x_1,\dots, x_n,h_1,\dots,h_{n-1}, h_{m(n)-1}$. In particular, $BV_2$ is generated by $x_1$, $x_2$, $h_1$ and $h_5$. The generators of $BV_2$ given by \citep{Brin2} are also four: $x_1$, $x_2$, $h_1$ and $(R,\sigma_1,R)$, where $R$ is the tree of depth 1.  

\medskip

\noindent
In general, if there is a known generating system $S_H$ for $H$, one can use the three mentioned results to obtain a generating system $$\left\{ (R,Id, \{s,Id,\dots, Id \}  ,R) \,|\, s\in S_H\right\} \cup \{x_1,\dots, x_n,h_1,\dots,h_{n-1}, h_{m(n)-1}\}.$$

%
%
%
%
%

\subsection{The specific case when $H$ is a parabolic subgroup.}
 
Important types of subgroups $H$ are parabolic subgroups. If $\Sigma$ is the set of Artin generators of $\mathcal{B}_n$, a \emph{standard parabolic subgroup} $A_X$ is the subgroup of $\mathcal{B}_n$ generated by a subset $X\subset S$. 
A \emph{parabolic subgroup} is defined as any conjugate of a standard parabolic subgroup.
To obtain a set of $2n$ generators in $BV_n(A_X)$, one shall slightly modify the proof of \autoref{lemma1:fg}:

\begin{lemma}\label{lemma4:fg}
Let $X$ be a subset of the Artin generators $\Sigma$ of $\mathcal{B}_n$. The generators $h_i$ in \autoref{cor:fg} such that $\sigma_i\in X$ can be expressed as a word using generators of type $e$, $h_j$ such that $\sigma_j \in \Sigma \setminus X$ and $g_i$ such that $\sigma_i\in X$.
\end{lemma}

\begin{proof}

As for \autoref{lemma1:fg}, let
\begin{align*}
w_i^{(k)} & :=\left(T(n),Id, \{\mu_j\}_{j=0}^{m(n)-1}, T(n)\right) \ \mbox{where} \ \begin{cases}  \mu_j= \sigma_i, \ \mbox{if} \ j =k,\\ \mu_j = Id, \ \mbox{otherwise}. \\
\end{cases}\\
z_i & :=\left(T(n),\sigma_i, \{\xi_j\}_{j=0}^{m(n)-1}, T(n)\right) \ \mbox{where} \ \begin{cases}
\xi_j =\sigma_i, \  \mbox{if} \ j \in \{i+1,\dots, n-1\},\\
\xi_j=Id, \ \mbox{otherwise}.\\
\end{cases}
\\
\end{align*} for $i \in \{1,\dots, n-1\}$.
Let also $I:=\{i \,|\,\sigma_i\in X \}$.

\medskip\noindent
\emph{Claim.} For any $i\in I$, the elements $z_i$ and $w_i^{(k)}$, with $0\leq k \leq i$, can be expressed by using generators of type~$e$, $h_j$ such that $j\not\in I$ and $g_{i'}$ such that $i'\in I$.

\medskip\noindent
For $k=0$, ${w_i^{(0)}}\equiv e_S^{-1} g_i e_S$, where $S$ is the spine tree of same depth as $T(n)$. 
We have also seen that if $x_i:= g_i  \left(w_i^{(i-1)}\right)^{-1}$, then $w_i^{(i)}\equiv g_i  x_i ^{-1}$. Since $z_i\equiv g_i  \left( w_i^{(0)}\right)^{-1} \left(w_i^{(1)}\right)^{-1}\cdots \left(w_i^{(i)}\right)^{-1}$, we have to prove that $w_i^{(k)}$ is generated as desired, for $1<k<i$.  

\smallskip
Suppose that the claim holds for $j<k<i$.
Denote by $r_p$ the element $\sigma_{p}\sigma_{p-1}\cdots \sigma_1$ for $p>1$ and  let $r'_p:=\sigma'_{p}\sigma'_{p-1}\cdots \sigma'_1$ where $\sigma'_q=\sigma_q$ if $q\in I$ and $\sigma'_q=Id$, if $q\not\in I$. 
Consider the element represented by $r(k):= z'_{k}z'_{k-1}\cdots z'_{1}$ for $0\leq k < i$, where $z'_q=z_q$ if $q\in I$ and $z'_q=h_q$, if $q\not\in I$. A representative of this element is $\left(T(n), r_k, \rho_k, T(n)\right)$ where $$\rho_k=\{Id, Id,r'_1,r'_2\dots, r'_{k},\overbrace{r'_{k+1},\dots, r'_{k+1}}^{n-k-1}, Id,\dots, Id\}.$$ Then we have that  $w_i^{(k)}\equiv r(k)  e_S^{-1} g_i e_S r(k)^{-1}$, for $0\leq k<i$. This finishes the proof of the claim. 

\medskip\noindent
Finally, as $h_{n-1}=z_{n-1}$, the result follows by induction using the product: \[ h_{i+1} w_i^{(i-1)}  z_i  h_{i+1}  z_{i}^{-1} {w_i^{(i-1)}}^{-1}  h_{i+1}^{-1}.\]

\end{proof} 

\begin{proposition}
Let $X$ be a subset of the  Artin generators $\Sigma$ of $\mathcal{B}_n$. Then $BV_n(A_X)$ is generated by $2n$ elements in $BV_n(A_X)$, namely the set $$\{x_0,\dots , x_{n-1}\} \cup \{h_i\,|\, \sigma_i \in \Sigma\setminus X\}\cup \{h_{m(n)-1} \}\cup \{g_i\,|\, \sigma_i \in X\}.$$
Similarly, $BV_n(\alpha^{-1} A_X \alpha)$, $\alpha\in \mathcal{B}_n$ , is generated by the following set of $2n$ elements: $$\{ h_\alpha^{-1}x_0h_\alpha,\dots , h_\alpha^{-1} x_{n-1} h_\alpha\} \cup \{ h_\alpha^{-1} h_ih_\alpha\,|\, \sigma_i \in \Sigma\setminus X\}\cup \{h_\alpha^{-1} h_{m(n)-1}h_\alpha \} \cup \{ h_\alpha^{-1} g_\alpha^{-1}g_i g_\alpha h_\alpha \,|\, \sigma_i \in X\},$$
where $h_\alpha:= \left(T(n), \alpha, \overrightarrow{Id}, T(n)\right)$, $g_\alpha:=\left(R,Id,\{\alpha,Id,\dots, Id\}, R\right)$ and $R$ is the $n$-ary tree of depth $1$.
\end{proposition}

\begin{proof}
For $A_X$, this a direct consequence of \autoref{remark:gen}, \autoref{lemma2:fg}, \autoref{lemma3:fg} and \autoref{lemma4:fg}.
As conjugacy defines a group isomorphism, it is easy to proof (see \autoref{remark:gen}) that $$h_\alpha^{-1} \{ x_0,\dots, x_{n-1},h_1,\dots, h_{m(n)-1} \} h_\alpha \cup  h_\alpha^{-1}g_\alpha^{-1}\{g_i\,|\, \sigma_i \in X\} g_\alpha h_\alpha$$ generates $BV_n(\alpha^{-1} A_X \alpha)$. This set can be reduced to 
$$h_\alpha^{-1} \{ x_0,\dots, x_{n-1},h_1,\dots,h_{n-1}, h_{m(n)-1} \} h_\alpha \cup h_\alpha^{-1}g_\alpha^{-1}\{g_i\,|\, \sigma_i \in X\} g_\alpha h_\alpha$$ by reproducing the proof of \autoref{lemma2:fg} with all elements conjugated by $h_\alpha$. We need to show that one does not need $h_i$ such that $\sigma_i\in X$ in this generating system. To prove that, we encourage the reader to rewrite the proof of \autoref{lemma4:fg}, doing the following conjugacy changes:

\medskip
\noindent

\begin{itemize}

\item Redefine $w_i^{(k)}:=h_\alpha^{-1}\left(T(n),Id, \{\mu_j\}_{j=0}^{m(n)-1}, T(n)\right) h_\alpha$ with  $ \mu_k= \alpha^{-1}\sigma_i\alpha$ and $\mu_j = Id$ when $j\neq k$.

\item Redefine $z_i:= h_\alpha^{-1}\left(T(n),\sigma_i, \{\xi_j\}_{j=0}^{m(n)-1}, T(n)\right) h_\alpha $ with $\xi_j =\alpha^{-1}\sigma_i \alpha$ when $j=i+1,\dots, n-1$ and $\xi_j=Id$ otherwise. 

\item Conjugate generators of type $e$ and $h$ by $h_\alpha$ and conjugate generators of type $g$ by $g_\alpha h_\alpha$.

\item Conjugate all $\sigma_q$ by $\alpha$.

\end{itemize}

\end{proof}

\subsection{Finite generation for \texorpdfstring{$\boldsymbol{BV_{n,r}(H)}$}{BVnrH}}

It is straightforward to prove that the groups $BV_{n,r}(H)$ are finitely generated by using the methods of Section~4.1 together with a result of \citep{Brown}, whose proof is included hereunder for completeness.

\begin{theorem}
\label{thm:F}
Let $n\geq 2$. Then $F_{n,r} \simeq \widehat{F_n}$ for every $r \geq 1$.
\end{theorem}

\begin{proof}

Every $V_{n,r}$ contains an isomorphic copy of $\widehat{V_n}$ in the following way: consider the set of $r$~roots of~$V_{n,r}$ and expand the last one by adding a final caret to the $r$-th root. We continue this process expanding the rightmost leaf of the resulting tree, and so on. The final result is an infinite right spine appended to the last root. This tree is invariant by any element of $F_{n,r}$, so the restriction map $F_{n,r} \rightarrow \widehat{F_n}$ is an isomorphism.
\end{proof}

\begin{theorem}
If $H \leq \mathcal{B}_n$ is finitely generated, the groups $BV_{n,r}(H)$ are finitely generated for all $n\geq2, r\geq 1$.
\end{theorem}

\begin{proof}
All the ideas applied in this section can be adapted to any $BV_{n,r}(H)$, getting similar generators of type $g$, $h$ and $e$. This is an easy but laborious exercise. In particular, the generators of type $e$ are elements in $F_{n,r}$, so we can replace them by the generators of $F_{n,r}$, which can be obtained via the isomorphism of \autoref{thm:F}.
\end{proof}

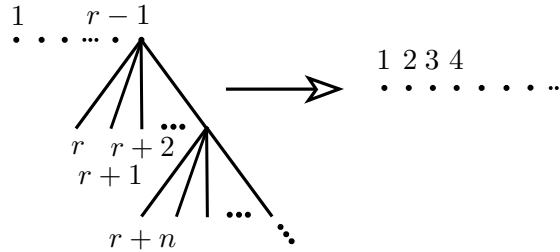
\begin{figure}[h]

\medskip

\centering
\labellist

\pinlabel $1$ at 4 60
\pinlabel $r-1$ at 29 60
\pinlabel $r$ at 19 27
\pinlabel $r+1$ at 27 20
\pinlabel $r+2$ at 35 27
\pinlabel $r+n$ at 35 4

\pinlabel $1$ at 95 49
\pinlabel $2$ at 101.5 49
\pinlabel $3$ at 107 49
\pinlabel $4$ at 113 49
\endlabellist
\includegraphics[scale=1.5]{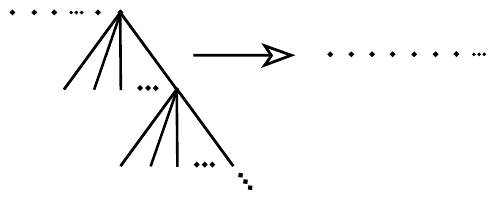}
\caption{The isomorphism between $F_{n,r}$ and $\widehat{F_n}$.}
\label{fig:browntheorem}
\end{figure}

\noindent
The next step on the study of these groups is to ask whether they are finitely presented.

\bigskip
\noindent
\textbf{Conjecture.} $BV_{n,r}(H)$ is finitely presented when $H$ is finitely presented.

\bigskip
Our guess is that this conjecture should be true, because it is likely that Brin and Dehornoy methods to find presentations for $BV_2$ can be extended to $BV_{n,r}$ (or at least to $BV_n$). Once a presentation for $BV_{n,r}$ is found, it should be possible to deal with white vertices. We have tried to find a presentation for $BV_{n,r}(H)$ using braided diagrams but we have not succeeded so far.

\bigskip

\noindent{\textbf{\Large{Acknowledgments}}} 

We thank the referee of this paper for a careful reading and a thoughtful report.

\smallskip

The first author wants to acknowledge financial support from the Spanish Ministry of Science and Innovation, through the ``Severo Ochoa Programme for Centres of Excellence in R\&D'' (SEV-2015-0554) and from the Spanish National Research Council, through the ``Ayuda extraordinaria a Centros de Excelencia Severo Ochoa'' (20205CEX001). He is also grateful to Javier Aramayona for conversations around the material in this note.

\smallskip

The second author was supported by the research grants MTM2016-76453-C2-1-P (financed by the Spanish Ministry of Economy and FEDER), US-1263032 (financed by the Andalusian Ministry of Economy and Knowledge and the Operational Program FEDER 2014--2020) and EP/S010963/1 (financed by the Engineering and Physical Sciences Research Council in UK). She wants to thank Heriot-Watt University and University of Burgundy, especially Alexandre Martin and Luis Paris, for the postdoc contracts that have allowed her to continue working on research, and in particular on this paper. She is always grateful to Juan Gonz\'alez-Meneses for his eternal patience and support. 

\smallskip

We dedicated this paper to the memory of Patrick Dehornoy, who left us too soon. His invaluable contributions to the study of braid groups and Garside theory will keep a part of him alive, but we will never forget his kindness, enthusiasm and good humour. Merci pour tout, Patrick.

\bibliography{Bib_Thompson}

\begin{thebibliography}{}

\bibitem[\protect\citename{Aroca, }2018]{Aroca}
J. Aroca. 2018.
\newblock {\em The conjugacy problem for {T}hompson-like groups}.
\newblock arXiv:1807.09503v2.

\bibitem[\protect\citename{Artin, }1947]{Artin1}
E. Artin. 1947.
\newblock {Theory of braids}.
\newblock {\em Ann. of Math.(2)}, {\bf 48}, 101--126.

\bibitem[\protect\citename{Bartholdi {\em et~al.}, }2003]{BGS}
L. Bartholdi, R.~I. Grigorchuk, \& Z. \v{S}uni\'{k}. 2003.
\newblock Branch groups.
\newblock Handb. Algebr., vol. 3.

\bibitem[\protect\citename{Belk \& Matucci, }2014]{BelkMatucci}
J. Belk, \& F. Matucci. 2014.
\newblock Conjugacy and dynamics in {T}hompson's groups.
\newblock {\em Geom. Dedicata}, {\bf 169}, 239--261.

\bibitem[\protect\citename{Brin, }2006]{Brin2}
M.~G. Brin. 2006.
\newblock The algebra of strand splitting. {II}: A presentation for the braid
  group on one strand.
\newblock {\em Internat. J. Algebra Comput.}, {\bf 16}(01), 203--219.

\bibitem[\protect\citename{Brin, }2007]{Brin1}
M.~G. Brin. 2007.
\newblock The algebra of strand splitting. I. A braided version of Thompson's
  group~V.
\newblock {\em J. Group Theory}, {\bf 10}(6).

\bibitem[\protect\citename{Brown, }1987]{Brown}
K.~S. Brown. 1987.
\newblock Finiteness properties of groups.
\newblock {\em Pages  45--75 of:} {\em Proceedings of the {N}orthwestern
  conference on cohomology of groups ({E}vanston, {I}ll., 1985)},  vol. 44.

\bibitem[\protect\citename{Cannon {\em et~al.}, }1996]{CFP}
J.~W. Cannon, W.~J. Floyd, \& W.~R. Parry. 1996.
\newblock Introductory notes on {R}ichard {T}hompson's groups.
\newblock {\em Enseign. Math.}, {\bf 42}, 215--256.

\bibitem[\protect\citename{Dehornoy, }2006]{Dehornoy}
P. Dehornoy. 2006.
\newblock The group of parenthesized braids.
\newblock {\em Adv. Math.}, {\bf 205}(2), 354--409.

\bibitem[\protect\citename{Fenn {\em et~al.}, }1996]{FRZ}
R. Fenn, D. Rolfsen, \& J. Zhu. 1996.
\newblock Centralisers in the braid group and singular braid monoid.
\newblock {\em Enseign. Math.}, {\bf 42}, 75--96.

\bibitem[\protect\citename{Grigorchuk, }1980]{Grigorchuk}
R.~I. Grigorchuk. 1980.
\newblock On {B}urnside's problem on periodic groups.
\newblock {\em Funktsional. Anal. i Prilozhen.}, {\bf 14}(1), 53--54.

\bibitem[\protect\citename{Higman, }1974]{Higman}
G. Higman. 1974.
\newblock {\em Finitely presented infinite simple groups}.
\newblock Notes on pure mathematics.
\newblock Dept. of Pure Mathematics, Dept. of Mathematics, I.A.S., Australian
  National University.

\bibitem[\protect\citename{L.~Bartholdi \& Nekrashevych, }2003]{BGN}
R.~I.~Grigorchuk L.~Bartholdi, \& V. Nekrashevych. 2003.
\newblock From fractal groups to fractal sets.
\newblock {\em Pages  25--118 of:} {\em Fractals in {G}raz 2001}.
\newblock Trends Math.
\newblock Birkh\"{a}user, Basel.

\bibitem[\protect\citename{Newman, }1942]{Newman}
M.~H.~A. Newman. 1942.
\newblock On theories with a combinatorial definition of ``equivalence''.
\newblock {\em Ann. of Math. (2)}, {\bf 43}, 223--243.

\bibitem[\protect\citename{Witzel \& Zaremsky, }2018]{WitzelZaremsky}
S. Witzel, \& M.~C.~B. Zaremsky. 2018.
\newblock Thompson groups for systems of groups, and their finiteness
  properties.
\newblock {\em Groups Geom. Dyn.}, {\bf 12}(1), 289--358.

\bibitem[\protect\citename{Zaremsky, }2018]{Zaremsky}
M.~C.~B. Zaremsky. 2018.
\newblock A user's guide to cloning systems.
\newblock {\em Topology Proc.}, {\bf 52}, 13--33.

\end{thebibliography}

\bigskip

\textit{ Julio Aroca, Departamento de  Matem\'aticas,
Universidad Aut\'onoma de Madrid and 
Instituto de Ciencias Matem\'aticas, 
28049 (Madrid), Spain.} \par
 \textit{E-mail address:} \texttt{\href{mailto:julio.aroca@icmat.es}{julio.aroca@icmat.es }} 

 \medskip

\textit{ Mar\'{i}a Cumplido, Department of Mathematics, Heriot-Watt University, Riccarton, EH14 4AS (Edinburgh), Scotland, UK.} \par
 \textit{E-mail address:} \texttt{\href{mailto:M.Cumplido@hw.ac.uk}{M.Cumplido@hw.ac.uk} }

\end{document}